\documentclass[a4paper,11pt,oneside]{amsart}

\usepackage[german, english]{babel}
\usepackage{graphicx}
\usepackage{microtype}
\usepackage{color}
\usepackage{xcolor}
\usepackage{amsmath}
\usepackage{amssymb}
\usepackage{amsfonts}
\usepackage{amsthm}
\usepackage{dsfont}
\usepackage{geometry}
\usepackage{media9}
\usepackage[T1]{fontenc}
\usepackage[utf8]{inputenc}
\usepackage{lmodern}
\usepackage{longtable}
\usepackage{textcomp}
\usepackage{booktabs}
\usepackage{caption}
\usepackage{subcaption}
\usepackage{setspace}
\usepackage{calc}
\usepackage{mathrsfs} %für noch mehr geschwungene Buchstaben
\usepackage{tabularx}
\usepackage[arrow, matrix, curve]{xy} %für kommutative Diagramme
\usepackage{lscape}
\usepackage{floatflt}
\usepackage{rotating}
\usepackage{siunitx}
\usepackage{cancel}
\usepackage{icomma}
\usepackage{tikz}									

\usetikzlibrary{shapes.geometric, arrows}	

\tikzstyle{startstop} = [minimum width=3cm, text width = 5cm, minimum height=1cm, text centered, draw=black]
\tikzstyle{arrow} = [thick, ->, >=stealth]

\usetikzlibrary{matrix}
\usetikzlibrary{patterns}
\usetikzlibrary{positioning}
\usetikzlibrary{decorations.pathmorphing}
\usetikzlibrary{cd}

\usepackage{xspace}

\usepackage[colorlinks]{hyperref}

\theoremstyle{definition}
\newtheorem{definition}{Definition}[section]

\newtheorem{exa}[definition]{Example}

\newtheorem*{que*}{Question}

\theoremstyle{remark}

\newtheorem{rem}[definition]{Remark}
\newtheorem{example}[definition]{Example}

\theoremstyle{plain}
\newtheorem{theorem}[definition]{Theorem}
\newtheorem*{theorem*}{Theorem}
\newtheorem{lemma}[definition]{Lemma}
\newtheorem{prop}[definition]{Proposition}

\newcommand\CCt{{\mathbb C}\{\!\!\{t\}\!\!\}}
\newcommand\RRt{{\mathbb R}\{\!\!\{t\}\!\!\}}
\newcommand\KK{{\mathbb K}}
\newcommand\RR{{\mathbb R}}
\newcommand\BShape{B}
\newcommand\cT{{\mathcal T}}

\DeclareMathOperator{\Trop}{Trop}

\newcommand\polymake{{\tt polymake}\xspace}

 \newcommand{\simplex}{
\coordinate (p0) at (0,0);
\coordinate (p1) at (1,0);
\coordinate (p2) at (0,1);
\coordinate (p3) at (2,0);
\coordinate (p4) at (1,1);
\coordinate (p5) at (0,2);
\coordinate (p6) at (3,0);
\coordinate (p7) at (2,1);
\coordinate (p8) at (1,2);
\coordinate (p9) at (0,3);
\coordinate (p10) at (4,0);
\coordinate (p11) at (3,1);
\coordinate (p12) at (2,2);
\coordinate (p13) at (1,3);
\coordinate (p14) at (0,4);
 \path (p0) node[circle, gray, fill, inner sep=1.5]{};
 \path (p1) node[circle, gray,fill, inner sep=1.5]{};
 \path (p2) node[circle, gray, fill, inner sep=1.5]{};
 \path (p3) node[circle, gray, fill, inner sep=1.5]{};
 \path (p4) node[circle, gray, fill, inner sep=1.5]{};
 \path (p5) node[circle, gray, fill, inner sep=1.5]{};
 \path (p6) node[circle, gray, fill, inner sep=1.5]{};
 \path (p7) node[circle, gray, fill, inner sep=1.5]{};
 \path (p8) node[circle, gray, fill, inner sep=1.5]{};
 \path (p9) node[circle, gray, fill, inner sep=1.5]{};
 \path (p10) node[circle, gray, fill, inner sep=1.5]{};
 \path (p11) node[circle, gray, fill, inner sep=1.5]{};
 \path (p12) node[circle, gray, fill, inner sep=1.5]{};
 \path (p13) node[circle, gray, fill, inner sep=1.5]{};
 \path (p14) node[circle, gray, fill, inner sep=1.5]{};
\draw[thin, gray] (0,0) -- (0,4);
\draw[thin, gray] (0,0) -- (4,0);
\draw[thin, gray] (4,0) -- (0,4);
 }

\title{A tropical count of real bitangents to  plane quartic curves}
\author{Alheydis Geiger}
\address[Alheydis Geiger]{Department of Mathematics, University of T\"{u}bingen, Germany}
\email{\href{mailto:alheydis.geiger@math.uni-tuebingen.de}{alheydis.geiger@math.uni-tuebingen.de}}
\author{Marta Panizzut}
\address[Marta Panizzut]{ Technische Universität Berlin, Chair of Discrete Mathematics/Geometry  }
\email{\href{mailto:panizzut@math.tu-berlin.de}{panizzut@math.tu-berlin.de}}


\begin{document}
\bibliographystyle{plain}
\setlength{\unitlength}{1cm}

\begin{abstract}
A smooth tropical quartic curve has seven tropical bitangent classes. Their  shapes can vary within the same combinatorial type of  curve. We study deformations of these shapes and we show that the conditions determined by Cueto and Markwig for lifting them to real bitangent lines are independent of the deformations. From this we deduce a tropical proof of  Plücker and Zeuthen's  count of the number of real bitangents to smooth plane quartic curves. 		

\end{abstract}
\maketitle
	
\section{Introduction}
\noindent
The number of bitangent lines to a smooth plane quartic curve is a classical result from the 19th century. Pl\"ucker proved in 1834 that such a curve in the complex projective plane has $28$ bitangents, \cite{Plue34}. 
Building on an extensive first count by Pl\"ucker \cite{Plue39}, Zeuthen proved that a smooth quartic curve can have either $4$, $8$, $16$ or $28$ real bitangents depending  on the topology of the underlying real curve in real projective plane, \cite{Zeu73}. 
In this paper we provide a count of real bitangents to a tropically smooth plane algebraic quartic curve using computations and techniques in tropical geometry.

Smooth tropical quartic curves can have exactly $7$ or infinitely many bitangent tropical lines grouped into $7$ equivalence classes modulo continuous translations that preserve bitangency. This was proven  by Baker et al. \cite{BLMPR16} using the theory of divisors  on tropical curves. Bitangents to non-smooth tropical quartics were investigated in \cite{LeLe18}. 

Questions about lifting tropical bitangents were first considered by by Chan and Jiradilok \cite{ChJi17} for $K_4$-curves. Len and Markwig \cite{LeMa19} showed that under certain  genericity conditions, each class lifts to four bitangent lines over the complex Puiseux series, reproducing the $28$ bitangents from Pl\"ucker's theorem for tropically smooth quartic curves. These four lifts can be realized by more than one representative in the class and with different multiplicities. 
Cueto and Markwig \cite{CueMa20} proved that each class lifts either zero or four times to real bitangent lines, and they remarked that these real lifts are always totally real, meaning that also the tangency points have real coordinates. 
As a consequence, every number of real bitangents appearing for a tropically smooth real quartic curve must be divisible by $4$. However, a tropical proof of why only the numbers $4$, $8$, $16$ and $28$ are observed was still open.
The main result of this paper closes this research gap by providing a tropical version of Pl\"ucker and Zeuthen's count.

\begin{theorem}\label{thm:Pluecker}
	Let $\Gamma$ be a generic tropicalization of a smooth quartic plane curve  defined over a real closed complete non-Archimedean valued field. Either  $1$, $2$, $4$ or $7$ of its bitangent classes admit a lift to real bitangents near the tropical limit. 	Every smooth quartic curve whose tropicalization is generic has either $4$, $8$, $16$ or $28$ totally real bitangents.
\end{theorem}

To prove this theorem, we relied on the classification of the combinatorial structure of the bitangent classes by Cueto and Markwig \cite{CueMa20}.
Smooth tropical quartics are dual to unimodular triangulations of the fourth dilation of the standard 2-simplex. The dual subdivision is also called the combinatorial type of a quartic. The shapes of bitangent classes of tropical quartic curves with the same combinatorial type do not need to be equal, as illustrated in Example \ref{ex:EFJ}.
This motivates the introduction of deformation classes, which collect for each bitangent class the varying shapes that appear within the same combinatorial type.
We provide a classification of deformation classes of tropical bitangents that uniquely depend on the combinatorics of the tropical curve.

\begin{theorem}\label{thm:classification}
	There are $24$ deformation classes of tropical bitangent classes to generic smooth tropical quartic curves modulo $S_3$-symmetry. Orbit representatives of their dual deformation motifs are summarized in Figure \ref{fig:defclasses}.  
\end{theorem}

The real lifting conditions, i.e., the conditions for admitting a lift to a bitangent  over a real closed field with a non-Archimedean valuation, determined by Cueto and Markwig~\cite{CueMa20}, provide us with local information on the number of lifts of each bitangent shape.

The cones in the secondary fan that induce unimodular triangulations parameterize tropically smooth quartics. As described in \cite{2GP21}, there is a hyperplane arrangement subdividing these cones, such that the quartics in each new chamber have fixed bitangent shapes. Theorem \ref{theorem:lifting} states that the lifting conditions do not depend on this refinement of the secondary fan.
This allows us to work with deformation classes, which are fixed by the combinatorial type of the curve.

Using \polymake \cite{polymake:2000}, we enumerate the deformation classes for every combinatorial type of tropical quartic curves. We then check the real lifting conditions obtaining the expected number of real bitangent lines. We assume the same genericity assumptions as in \cite{ LeMa19, CueMa20} and we discuss them further in Section \ref{sec:prelim}. In particular, they include  the smoothness of the tropicalized curve. 
Note that the count over the reals follows from working over a real closed field due % 
 to Tarski-Seidenberg Transfer Principle \cite[Theorem 1.4.2]{Basu:2006}. Real tropicalization goes back to work of Maslov \cite{Maslov}  and the study of logarithmic limit sets of (real) algebraic varieties, see also \cite{Aless13}.

This paper is organized as follows. In Section~\ref{sec:prelim}, we introduce the main definitions and we report the classification of tropical bitangent classes and their lifting conditions, as introduced in \cite{CueMa20}. We assume that the reader is familiar with basic definitions and results on tropical curves and regular triangulations. We refer to \cite{MS15, DLRS10} for further details.  Deformation classes  are defined and classified in Section~\ref{sec:classification}. In Section~\ref{sec:lifting}, we study their real lifts and prove Theorem \ref{thm:Pluecker}. The proof is based on the enumeration of deformation classes in \polymake.  The technical description of the algorithms and their implementation can be found in \cite{2GP21}. The proofs for the classification are always constructed by providing details of few cases, and then summarizing the main ideas of the remaining ones.  More examples and figures are collected in the Appendix \ref{sec:Appendix} in order to provide further geometric intuition of definitions and proofs. 
\medskip 

\textbf{Acknowledgments.} We are very grateful to Hannah Markwig, Angelica Cueto and Michael Joswig for valuable discussions on the topic and for their comments on earlier versions of this work.
We thank Hannah Markwig and Angelica Cueto for allowing us to include figures from their paper \cite{CueMa20}. We thank Hannah Markwig, Sam Payne and Kris Shaw for telling us about their current project and allowing us to mention it in our paper.  
The first author is funded by a PhD scholarship from the Cusanuswerk e.V..
This work is a contribution to the SFB-TRR 195 'Symbolic Tools in Mathematics and their Application' of the German Research Foundation (DFG).

\section{Tropical quartic curves and their bitangents}\label{sec:prelim}
\noindent
A plane quartic curve $V(f)$  is the zero set of a polynomial of degree four 
%\[
\begin{align}
f(x,y) =  \ & a_{00}   + a_{10} x +a_{01}y + a_{20} x^2 +  a_{11}xy + a_{02}y^2 + a_{30} x^3 + a_{21}x^2y  \\  \nonumber 
&+a_{12}xy^2 + a_{03}y^3 + a_{40} x^4+ a_{31}x^3y + a_{22}x^2y^2 + a_{13}xy^3 + a_{04} y^4. 
\end{align}\label{eq:polynomial}
%\]
We consider the tropicalization of curves defined over a  real closed complete non-archimedean valued field $\KK_{\RR}$ and its algebraic closure $\KK$.  The main examples are the fields of Puiseux series $\RRt$ and $\CCt$. For the tropicalization of curves, we use the \emph{max}-convention. We write $\lambda_{ij}$ for the valuations of the coefficients of the polynomial, i.e., $\lambda_{ij} = \text{val}(a_{ij})$.  Then $\Trop(V(f))$ is the tropical curve define by the tropical polynomial with coefficients $-\lambda_{ij}$.

We assume that the Newton polygon of $f$ is the fourth dilation  of the standard $2$-dimensional simplex $4 \Delta_2$. The polygon $4\Delta_2$ contains $15$ lattice points $p_{ij}$ corresponding to the monomials $x^i\, y^j$ of $f$. By duality, the combinatorics of the tropical curve $\Gamma = \Trop(V(f))$ is determined by the subdivision $\cT$ (of the lattice points) of $4\Delta_2$ induced by the coefficients $-\lambda_{ij}$. We use the notation $\cdot^{\vee}$ to refer to the dual of a vertex or an edge of $\Gamma$ in $\cT$, and viceversa. We only consider smooth tropical plane quartic curves, so the subdivisions of $4\Delta_2$ are unimodular triangulations. In particular, all lattice points in $4\Delta_2$ are vertices in the triangulation $\cT$.  The set of points in $\RR^{15}$ inducing the same subdivision is a relative open cone called the \emph{secondary cone} and denoted $\Sigma(\cT)$. For a point $c \in \Sigma(\cT)$, we use the notation $\Gamma_{c}$ to indicate the tropical quartic curve defined by the tropical polynomial with coefficients given by the coordinates of $c$.

Regular unimodular triangulations of $4 \Delta_2$ have been enumerated by Brodsky et al. in~\cite{BJMS15}. They counted $1278$ orbits of combinatorial types under the action of the symmetric group $S_3$. This group acts on the homogenization of the lattice points of $4\Delta_2$ and on the corresponding monomials of the polynomial $f$ and its tropicalization Trop$(f)$. 

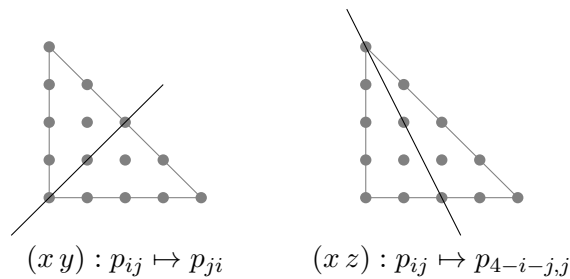
\begin{figure}[h]
	\begin{tikzpicture}[scale=0.5]
		\simplex
		\draw[thin, black] (-1,-1) -- (3,3);
		\node [below] at (2,-1) {$(x\,y):p_{ij} \mapsto p_{ji}$};
	\end{tikzpicture}
	\quad \quad
	\begin{tikzpicture}[scale=0.5]
		\simplex
		\draw[thin, black] (-0.5,5) -- (2.5,-1);
		\node [below] at (2,-1) {$(x\,z):p_{ij} \mapsto p_{4-i-j,j}$};
	\end{tikzpicture}
	\caption{Actions of generators of $S_3$ on $4\Delta_2$.}
\end{figure}
A tropical line $\Lambda$ is \emph{bitangent} to a smooth tropical plane quartic curve $\Gamma$ if their intersection $\Lambda \cap \Gamma$ has two components with stable multiplicity $2$, or one component with stable multiplicity $4$. See  \cite[Section 3.6]{MS15} for an introduction to stable intersections. We always assume that tropical bitangent lines are non-degenerate, i.e., each tropical line consists of a vertex and three adjacent rays with directions $-e_1$, $-e_2$ and $e_1+e_2$ given by the standard basis of $\RR^2$. For an impression of tropical bitangent lines, see Figure \ref{fig:EFJshapechange}.

A tropical quartic curve $\Gamma$ has exactly $7$ or infinitely many tropical bitangents. The collection of bitangents can be grouped into $7$ equivalence classes modulo continuous translations that preserve bitangency as shown in \cite{BLMPR16}.  Formally, the \emph{tropical bitangent classes} of a tropical quartic $\Gamma$ are defined as the connected components of $\RR^2$ containing the vertices of tropical bitangents in the same equivalent class. Up to $S_3$-symmetries, they refine into $41$ \emph{shapes of tropical bitagent classes}  given by coloring the points in the class belonging to $\Gamma$, see Figure \ref{fig:Fig6}. Since a non-degenerate tropical line is determined by its vertex, the bitangent classes formally live in the dual plane. They are connected polyhedral complexes, which are also min-tropical convex sets \cite[Theorem 1.1, Corollary 3.3]{CueMa20}. To improve the visualization, we draw the curve and its bitangent classes on the same plane. 

\begin{figure}[ht] 
	\includegraphics[width=10cm]{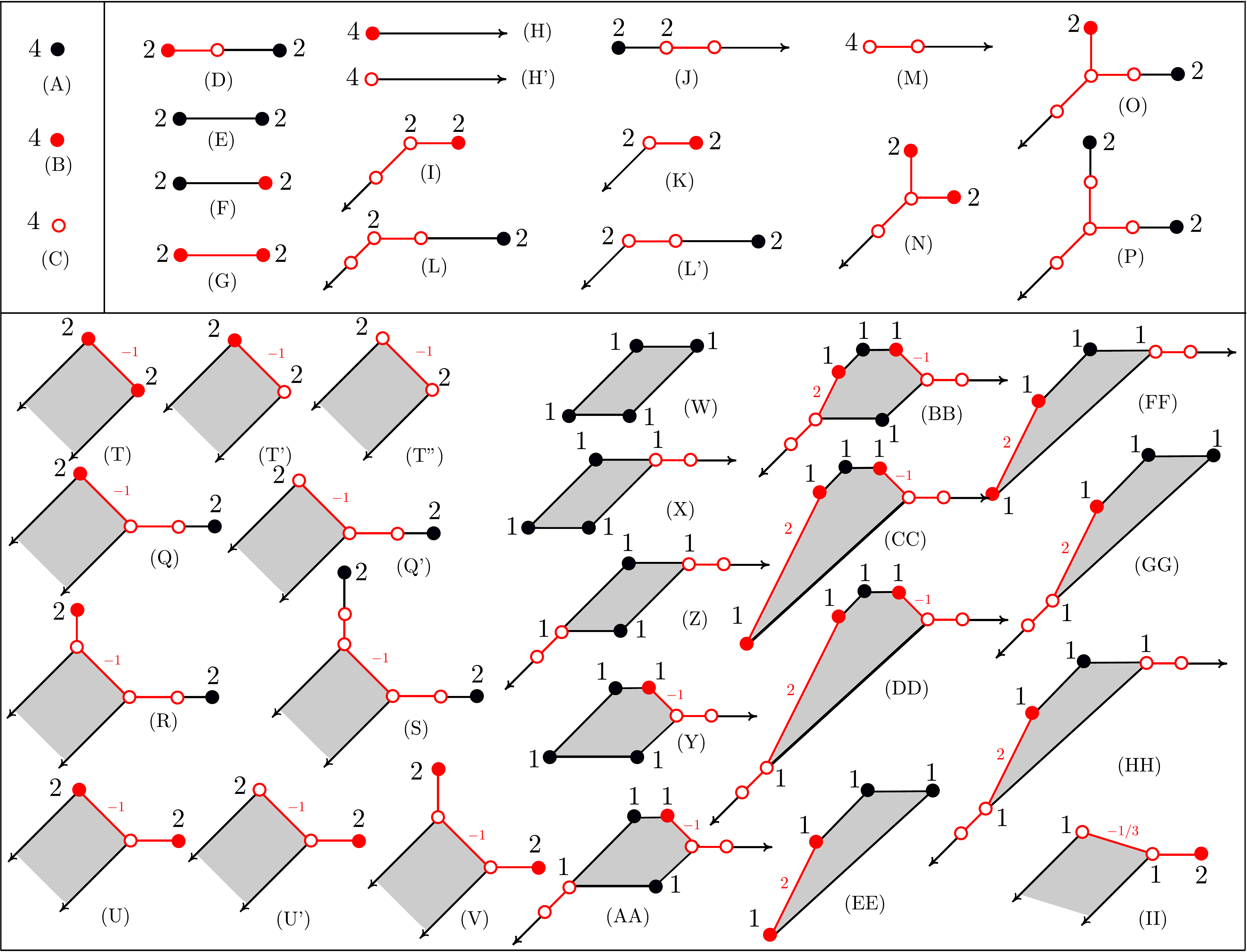}
	\caption{ Shapes of bitangent classes on smooth quartics. The black numbers above the vertices indicate the lifting representatives in each class and their lifting multiplicities. Red vertices or line segments are contained in the quartic curve, a red vertex filled with white coincides with a vertex of the quartic curve. Figure taken from \cite[Figure 6]{CueMa20}.}\label{fig:Fig6}
\end{figure}

As remarked by Cueto and Markwig \cite{CueMa20}, the shapes of tropical bitangent classes of a smooth tropical quartic curve $\Gamma$  impose combinatorial constraints on the regular unimodular triangulation of $4\, \Delta_2$ dual to $\Gamma$.  More precisely, the existence of a representative of a certain shape determines specific subcomplexes that must be contained in the triangulation, see Figure \ref{fig:Fig8}. Such a subcomplex is only a necessary condition for the presence of its corresponding tropical bitangent shape. Given a smooth tropical quartic curve $\Gamma$, the shapes of its bitangent classes are not fully determined by the combinatorial type of $\Gamma$, but they also depend on the length of the edges, as the following example illustrates.  This motivates us to introduce  deformation classes of tropical bitangents in the following section.

\begin{example} \label{ex:EFJ}
	We consider the quartic curve dual to the triangulation $\cT$ in Figure \ref{fig:exampletriangulation}. The colored subcomplex of $\cT$ corresponds to shapes (E), (F) and (J) in the classification in Figure \ref{fig:Fig6}. 		
	Let $\mathbf{a}$ denote the coefficient vector of an algebraic curve of degree $4$ with entries ordered  as in \eqref{eq:polynomial}.
	\begin{figure}[h]
		\centering
		\begin{subfigure}[b]{0.125\textwidth}
			\includegraphics[width=1\textwidth]{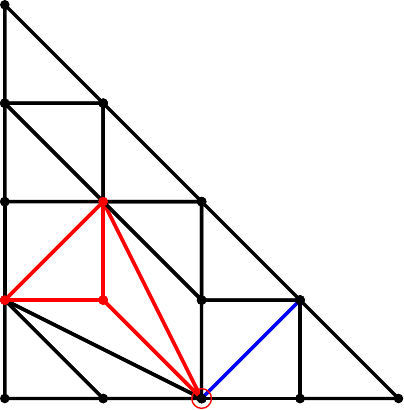}
			\caption{Triangulation $\cT$}\label{fig:exampletriangulation}
		\end{subfigure}
		\begin{subfigure}[b]{0.28\textwidth}
			\includegraphics[width=1\textwidth]{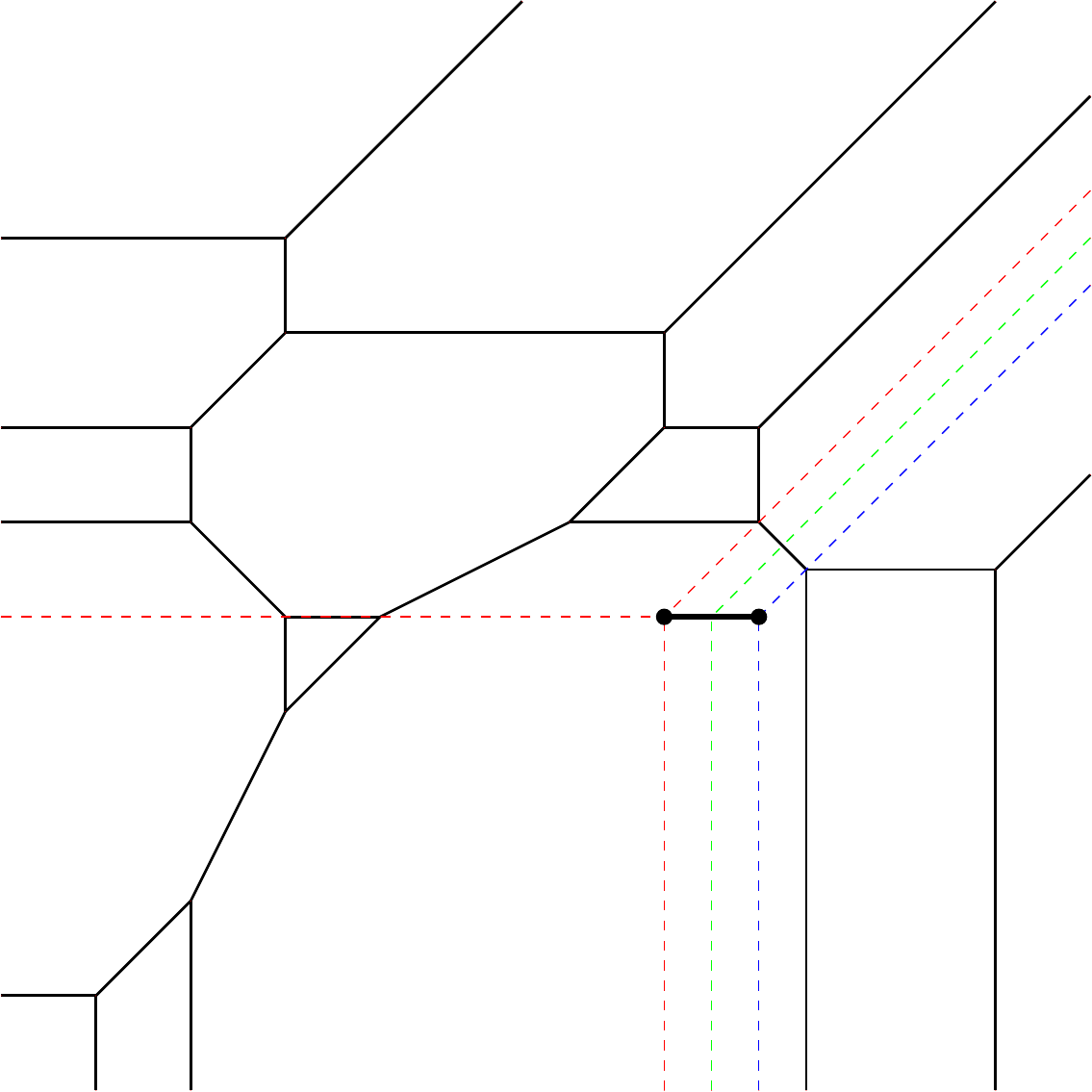}
			\caption{Shape (E)}\label{fig:shapeE}
		\end{subfigure}
		\begin{subfigure}[b]{0.28\textwidth}
			\includegraphics[width=1\textwidth]{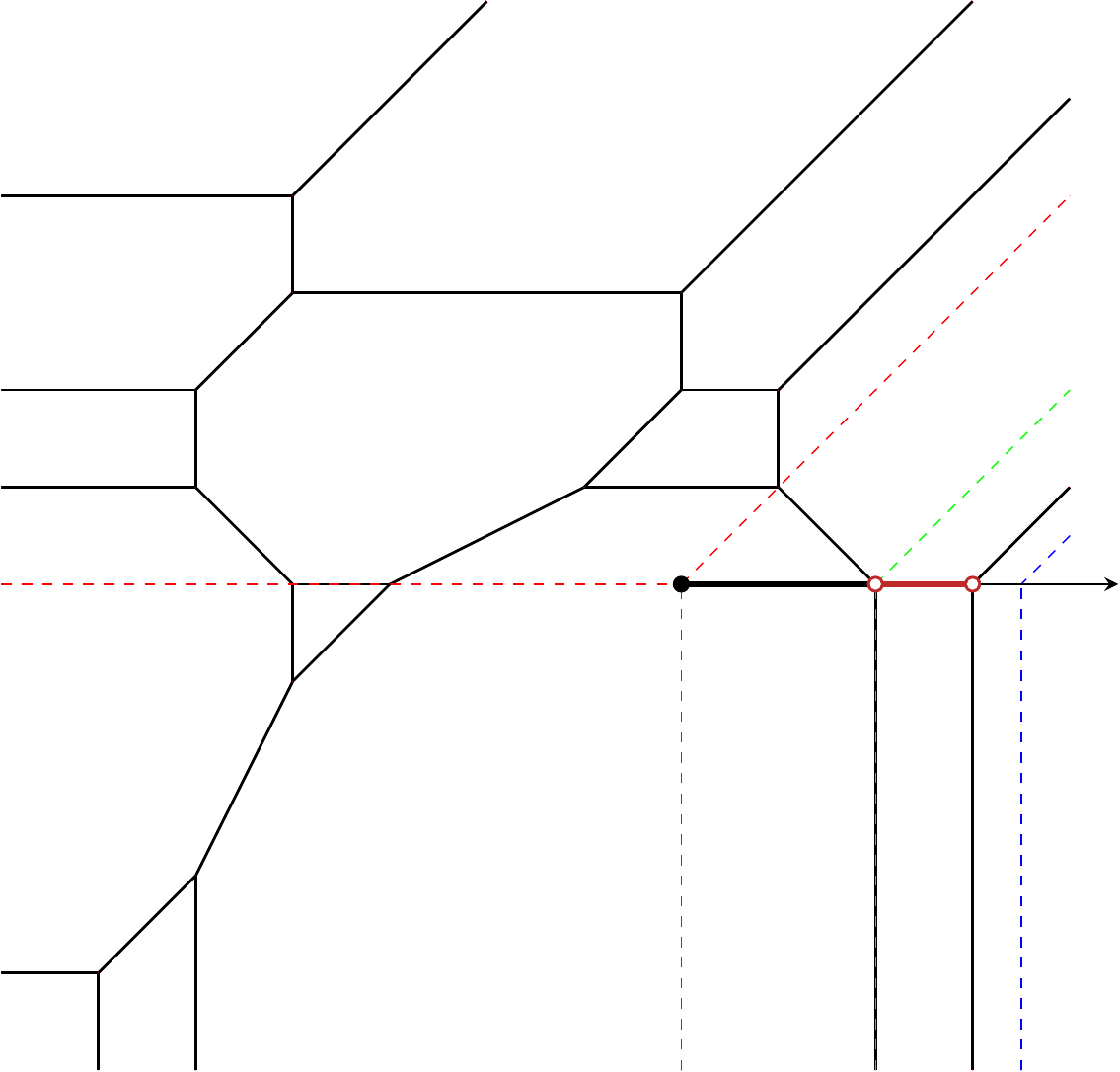}
			\caption{Shape (J)}\label{fig:shapeJ}
		\end{subfigure}
		\begin{subfigure}[b]{0.28\textwidth}
			\includegraphics[width=1\textwidth]{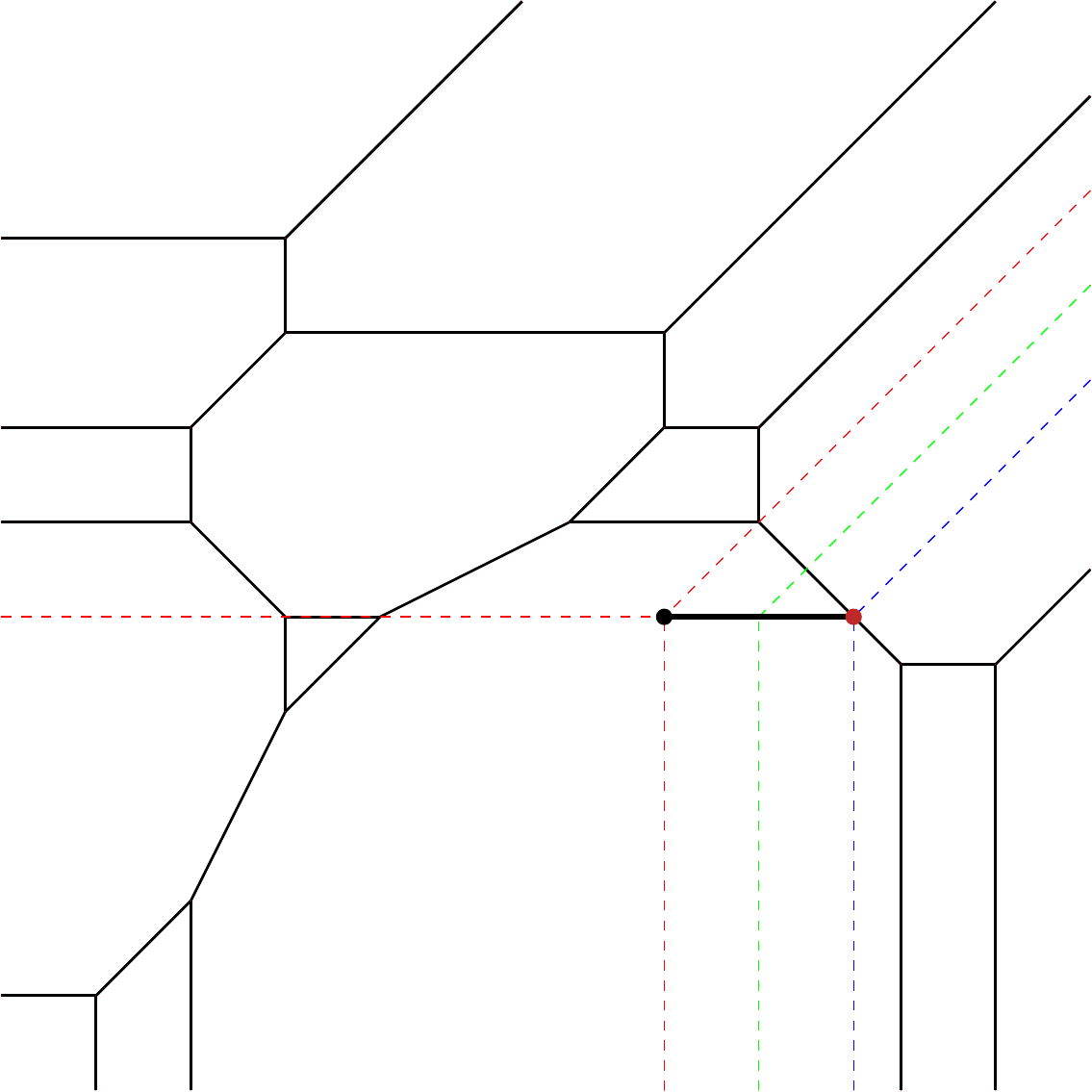}
			\caption{Shape (F)}\label{fig:shapeF}
		\end{subfigure}
		
		\caption{The dual triangulation does not fix the shapes of the bitangent classes, since they can change when choosing different edge lengths for the curve.}\label{fig:EFJshapechange}
	\end{figure}

We remark that there exists an equivalence class of bitangents of $\Gamma$ which adopts different shapes for different choices of $\lambda = \text{val}(\mathbf{a}) \in \Sigma(\cT)$.
 We observe shape (E) when choosing $\lambda_1 =(0,5,5,9,8,5,6.5,9,9,4,2,7,8,7,1)$, see Figure \ref{fig:shapeE}.
 For $\lambda_2 =(0,5,5,9,8,5,6,9,9,4,2,7,8,7,1)$ we obtain shape (J) as in Figure \ref{fig:shapeJ}.
 Figure~\ref{fig:shapeF} shows shape (F), which appears for $\lambda_3 =(0,5,5,9,8,5,5.5,9,9,4,1,7,8,7,1)$.

\end{example}

Let $\Gamma$ be a smooth tropical quartic curve and $V(f)$ a  smooth plane quartic curve defined over $\KK$ such that $\text{Trop}(V(f)) = \Gamma$. Following \cite[Definition 2.8]{LeMa19},  we say that a tropical bitangent $\Lambda$ with tangency points $P$ and $P'$ lifts over $\KK$ if there exists a bitangent $\ell$ to $V(f)$ defined over $\KK$ with tangency points $p$ and $p'$ such that 
\[
\text{Trop}(V(f)) = \Gamma, \ \ \text{Trop}(\ell) = \Lambda, \ \ \text{Trop}(p) = P, \ \ \text{and} \ \ \text{Trop}(p')=P'.
\] 

The number of such $\ell$ is the \emph{lifting multiplicity} of $\Lambda$. In an equivalence class, the number of tropical bitangents which lift and their lifting multiplicities can be one, two or four, see \cite[Theorem 4.1]{LeMa19}. 

Similarly, we are  is interested in the number of \emph{real} lifts, that is, the lifting multiplicity when $V(f)$ and $\ell$ are defined over $\KK_{\RR}$. In this case, as shown by Cueto and Markwig \cite{CueMa20}, every bitangent class of a given shape has either zero or four lifts to real bitangents. Moreover, the existence of a lift uniquely depends on the signs $s_{ij}$ of the coefficients $a_{ij}$ of the polynomial $f$. We report in Table \ref{tab:tab11} the summary of these conditions. 

Lifting problems were studied under the genericity constraints explained in \cite[Remark~2.10]{CueMa20}, which also apply in our results next to the  assumptions that the tropical curve $\Gamma$ is smooth and the tropical bitangent lines are non-degenerate. 
The assumption that if the tropical curve $\Gamma$ contains a vertex adjacent to three bounded edges with directions $-e_1$, $-e_2$ and $e_1+ e_2$, the shortest of these edges is unique will be particularly relevant in our analysis of the lifting conditions of shape (C).

\begin{center}
	\begin{tabular}{c|c}
	Shape & Lifting conditions \\
	\hline
	(A) & $(-s_{1v}s_{1,v+1})^i s_{0i}s_{22}>0$ and $(-s_{u1}s_{u+1,1})^js_{j0}s_{22}>0$ \\
	\hline
	(B) & $(-s_{1v}s_{1,v+1})^{i+1} s_{0i}s_{21}>0$ and $(-s_{21})^{j+1}s_{31}^js_{1v}s_{1,v+1}s_{j0}>0$ \\
	\hline
	(C) & $(-s_{11}s_{12})^is_{0i}s_{20}>0$ and $(-s_{21}s_{12})^k s_{k,4-k}s_{20}>0$ if j=2\\
	 & $(\!-s_{11}\!)^{i+1}s_{12}^is_{21}s_{0i}s_{j0}\!>\!0$ and $(\!-s_{21}\!)^{k+1}s_{12}^ks_{11} s_{k,4-k}s_{j0}\!>\!0$ if j=1,3\\
 	\hline
	(H),(H') &  $(-s_{1v}s_{1,v+1})^{i+1} s_{0i}s_{21}>0$ and $-s_{21}s_{1v}s_{1,v+1}s_{40}>0$ \\
	\hline
	(M) & $(-s_{1v}s_{1,v+1})^{i+1} s_{0i}s_{21}>0$ and $s_{31}s_{1v}s_{1,v+1}s_{30}>0$\\
	\hline
	(D) & $(-s_{10}s_{11})^is_{0i}s_{22}>0$\\
	\hline
	(E),(F),(J) & $(-s_{1v}s_{1,v+1})^{i+1} s_{0i}s_{20}>0$\\
	\hline
	(G) & $(-s_{10}s_{11})^is_{0i}s_{k,4-k}>0$\\
	\hline
	(I),(N) &$-s_{10}s_{11}s_{01}s_{k,4-k}>0$ \\
	\hline
	\shortstack{(K),(T),(T'),(T''), \\(U),(U'),(V)}& \raisebox{0.2cm}{$s_{00}s_{k,4-k}>0$} \\
	\hline
	(L),(O),(P) & $-s_{10}s_{11}s_{01}s_{22}>0$ \\
	\hline
	\shortstack{(L'),(Q),(Q'), \\ (R),(S) } & \raisebox{0.2cm}{ $s_{00}s_{22}>0$}\\
	\hline
	rest & no conditions
	\end{tabular}
	\captionof{table}{The real lifting conditions of the bitangent shapes in their identity positions as determined in \cite[Table 11]{CueMa20}.}\label{tab:tab11}

\end{center}

We conclude this section by fixing some notation and conventions. We follow the labeling of shapes of  bitangent classes  and the color patterns introduced in \cite{CueMa20} by Cueto and Marking,
 see Figures \ref{fig:Fig6} and \ref{fig:Fig8}.  The group $S_3$ acts  on bitangent shapes and their dual subcomplexes. We refer to the bitangent shapes and subcomplexes in Figures \ref{fig:Fig6} and \ref{fig:Fig8} as in  \textit{identity position}. 
We indicate the different elements in the orbit of a bitangent shape not in identity position by adding the element of $S_3$ acting on it to the index of the shape. For example, the notation (B)$_{(xy)}$ means that the bitangent class has shape (B) with dual complex  given by the action of $(x\,y)$ on  the complex in identity position.

\begin{figure}[h]
	\centering
	\includegraphics[width=10cm]{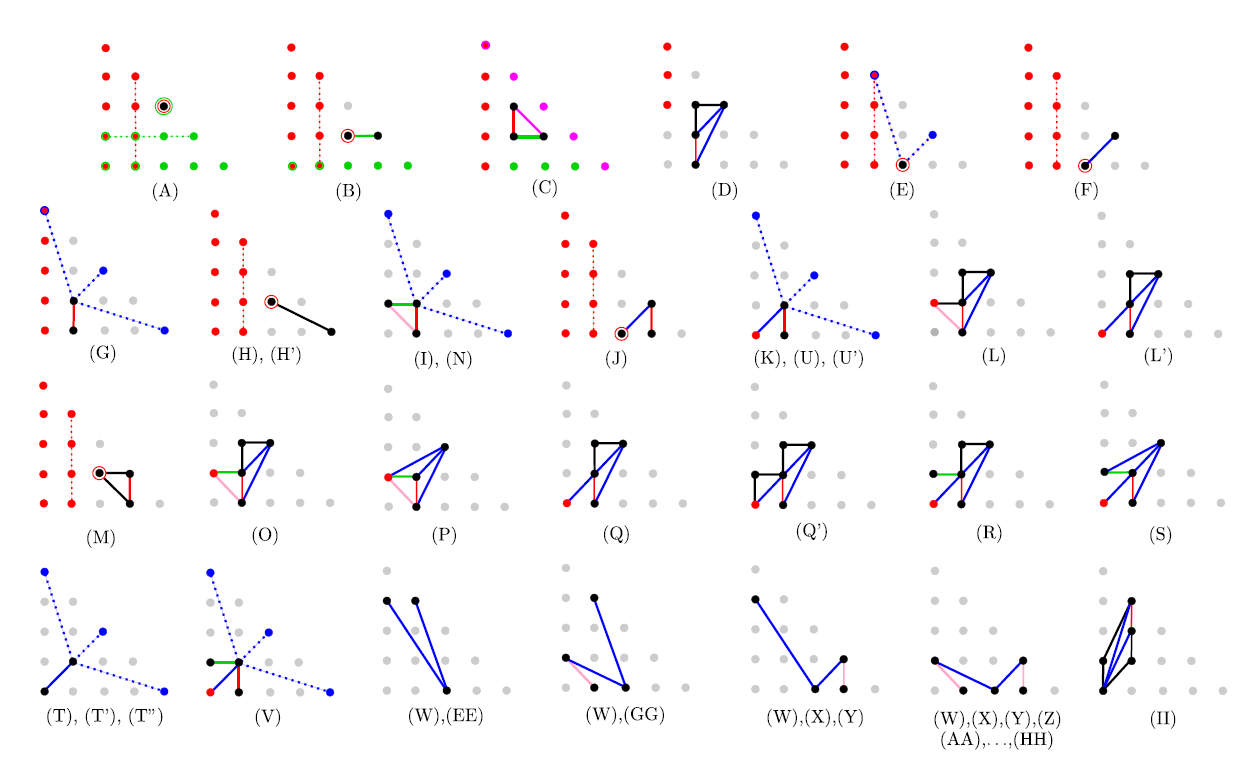}
	\caption{ Dual bitangent motifs of bitangent classes. The color coding is explained in \cite[Remark 4.13]{CueMa20}. Figure taken from \cite[Figure 8]{CueMa20}.}\label{fig:Fig8}
\end{figure}

\section{Deformation classes of tropical bitangents} \label{sec:classification}
\noindent
In this section we introduce deformation classes of tropical bitangents and we classify them. Our terminology is inspired by the one introduced in \cite{PaVi19, JPS21} for the classification of the combinatorial positions of tropical lines on cubic surfaces. 
\begin{definition}\label{def:motif}
Let $\Gamma$ be a tropical smooth quartic curve with dual triangulation $\cT$, and let $\BShape$ be a bitangent class of $\Gamma$ of a fixed shape. The \emph{dual bitangent motif} of $(\Gamma, \BShape)$ is the subcomplex of $\cT$ that is fully determined by the shape of $\BShape$.  Dual bitangent motifs  are classified in  \cite[Figure 8]{CueMa20}.
\end{definition}

In pictures of dual bitangent motifs, we use the same color coding as Cueto and Markwig \cite{CueMa20}. Black and colored solid edges must be part of the triangulations, while dotted ones represent possible edges, of which one must occur. Black vertices are always present, while the colored ones are endpoints of colored possible edges or they form a triangle with an edge of the same color. Different colors correspond to different types of tangencies, see \cite[Remark 4.13]{CueMa20}.

The summary of the dual bitangent motifs as shown in Figure \ref{fig:Fig8} is very condensed, and, as a consequence, for some figures with dotted edges not all combinations lead to the assigned bitangent shape. For example, when choosing the red and green edges $\overline{p_{10}p_{11}}$ and $\overline{p_{01}p_{11}}$ in the picture of (A), the shape, which will occur, will be either (S) or (P), but never (A). Therefore, the dual bitangent motifs will be subdivided more in the following classification of the deformation classes.

\begin{definition}\label{def:defclass}
Given a tropical quartic $\Gamma_{c}$ with dual triangulation $\cT$, $c\in\Sigma(\cT)$, and a tropical bitangent class $B$, we say that a tropical bitangent class $\BShape'$ is in the same \emph{deformation class} as $\BShape$ if the following conditions are satisfied:
\begin{itemize}
\item[$\triangleright$] There exists $\Gamma_{c'}$ with $c' \in \Sigma(\cT)$ having $B'$ as one of its bitangent classes.
\item[$\triangleright$] There is a continuous deformation from $\Gamma_{c}$ to $\Gamma_{c'}$ given by a path in the secondary cone $\Sigma(\cT)$ from $c$ to $c'$ that induces $\BShape$ to change to $\BShape'$.
\end{itemize} 
We use the notation $B_{\omega}$ to indicate the deformation of $B$ in $\Gamma_{\omega}$ for ${\omega}$ in the path. 
Given a unimodular triangulation $\cT$ of $4\Delta_2$ and a dual quartic curve $\Gamma$, let $\mathcal{D}$ be the deformation class of one of its seven bitangent classes. The \emph{dual deformation motif of $(\cT,\mathcal{D})$} is the union of the dual bitangent motifs of all shapes belonging to bitangent classes in  $\mathcal{D}$.
\end{definition}

We label deformation classes using the letters of the shapes of tropical bitangents. In Example \ref{ex:EFJ} we saw a deformation class (E F J). If the class contains the image of shapes under the action of an element $\sigma \in S_3$, we use the notation $+\sigma$. 

\begin{rem} \label{rem:dualmotifs}
Each smooth tropical quartic $\Gamma$ has $7$ deformation classes, and it follows from the definition that they only depend on the dual triangulation $\cT$ of $\Gamma$.  Changing the coefficients defining $\Gamma$ in the secondary cone $\Sigma(\cT)$ induces a variation in the shapes of the tropical bitangents within the deformation class.
\end{rem}

\begin{figure}[h]
	\centering
\begin{subfigure}{.21\textwidth}
	\centering
	\begin{subfigure}{.45\textwidth}
		\centering
		\includegraphics[width=0.9\textwidth]{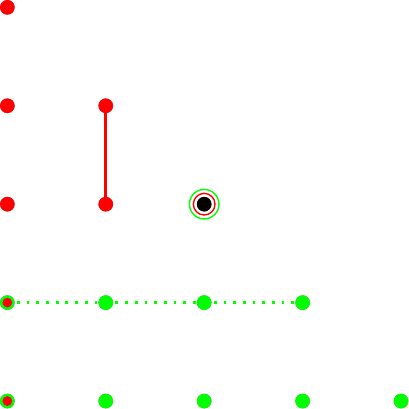}
	\end{subfigure}	\begin{subfigure}{.45\textwidth}
		\centering
		\includegraphics[width=0.9\textwidth]{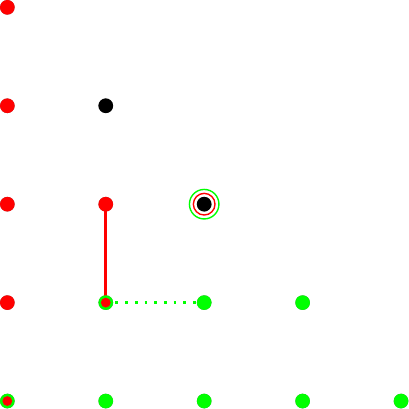}
	\end{subfigure}
	\caption*{(A)}
\end{subfigure}
\begin{subfigure}{.10\textwidth}
	\centering
	\includegraphics[width=0.9\textwidth]{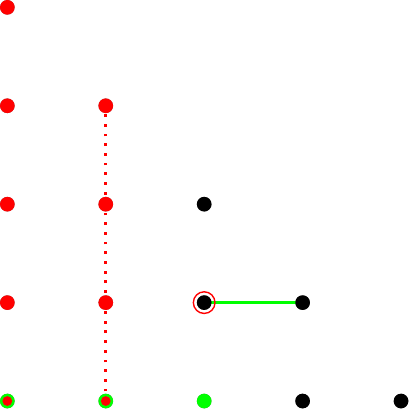}
	\caption*{(B)}
\end{subfigure}
\begin{subfigure}{.10\textwidth}
	\centering
	\includegraphics[width=0.9\textwidth]{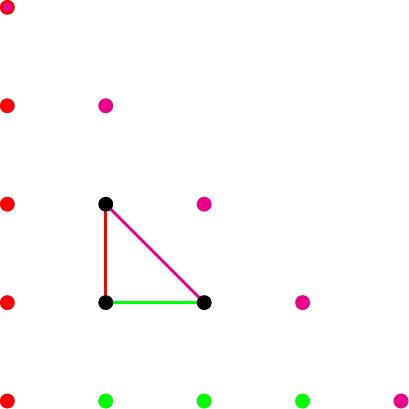}
	\caption*{(C)}
\end{subfigure}
\begin{subfigure}{.10\textwidth}
	\centering
	\includegraphics[width=0.9\textwidth]{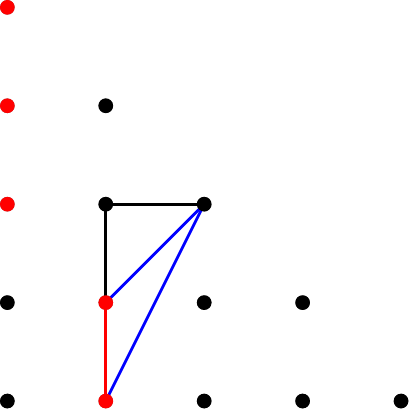}
	\caption*{(D)}
\end{subfigure}
\begin{subfigure}{.10\textwidth}
	\centering
	\includegraphics[width=0.9\textwidth]{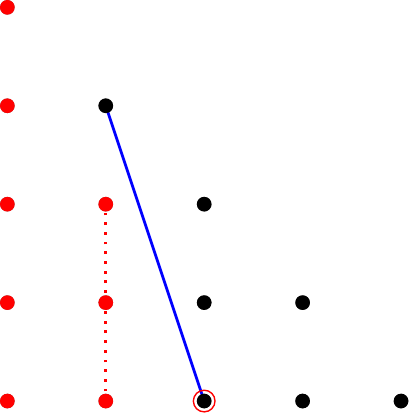}
	\caption*{(E)}\label{fig:(E)}
\end{subfigure}
\begin{subfigure}{.33\textwidth}
	\centering
	\begin{subfigure}{.3\textwidth}
		\centering
		\includegraphics[width=0.9\textwidth]{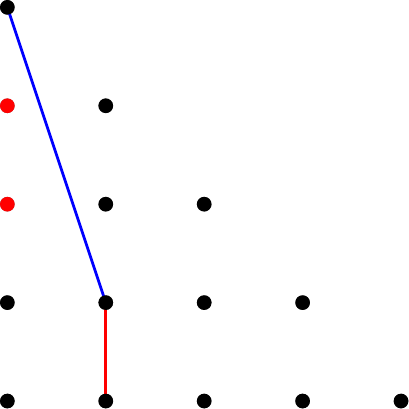}
	\end{subfigure}
	\begin{subfigure}{.3\textwidth}
		\centering
		\includegraphics[width=0.9\textwidth]{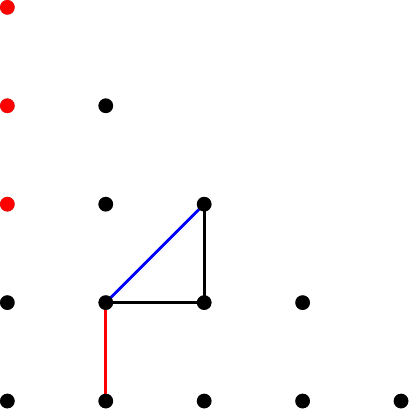}
	\end{subfigure}
	\begin{subfigure}{.3\textwidth}
		\centering
		\includegraphics[width=0.9\textwidth]{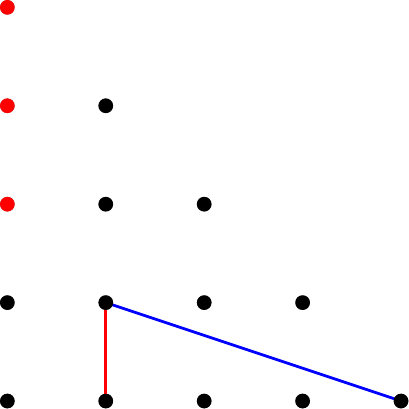}
	\end{subfigure}	
	\caption*{(G)}
\end{subfigure}
\begin{subfigure}{.10\textwidth}
	\centering
	\includegraphics[width=0.9\textwidth]{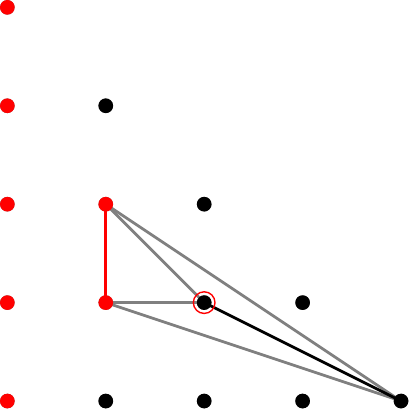}
	\caption*{(H)}
\end{subfigure}
\begin{subfigure}{.10\textwidth}
	\centering
	\includegraphics[width=0.9\textwidth]{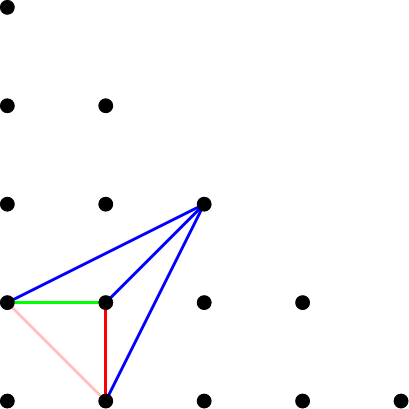}
	\caption*{(P)}
\end{subfigure}
\begin{subfigure}{.10\textwidth}
	\centering
	\includegraphics[width=0.9\textwidth]{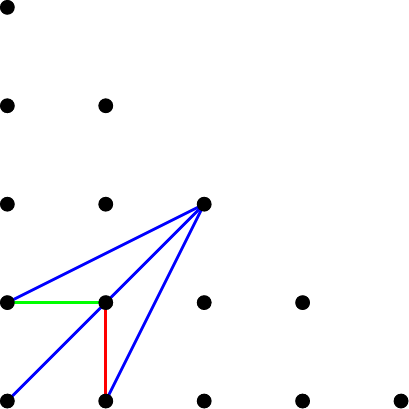}
	\caption*{(S)}
\end{subfigure}
\begin{subfigure}{.10\textwidth}
	\centering
	\includegraphics[width=0.9\textwidth]{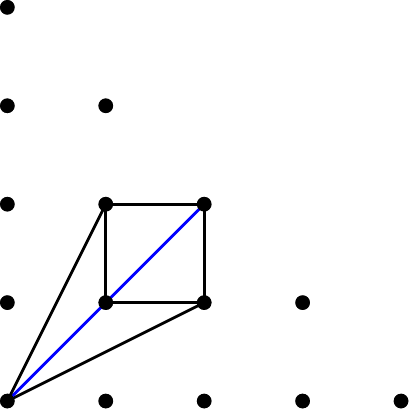}
	\caption*{(T)}
\end{subfigure}
\begin{subfigure}{.10\textwidth}
	\centering
	\includegraphics[width=0.9\textwidth]{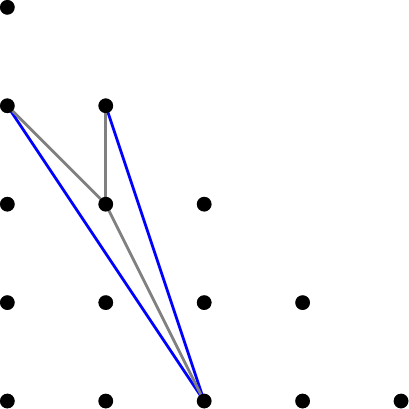}
	\caption*{(W)}
\end{subfigure}
\begin{subfigure}{.10\textwidth}
	\centering
	\includegraphics[width=0.9\textwidth]{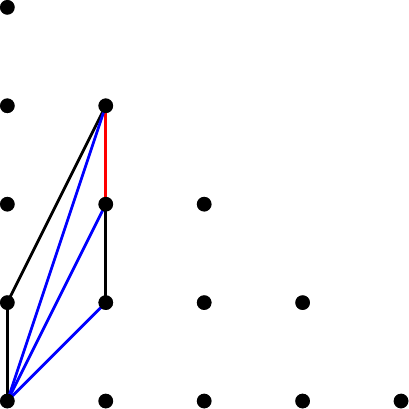}
	\caption*{(II)}
\end{subfigure}
\caption{Dual bitangent motifs of bitangent classes with constant shape in their deformation class.}\label{fig:subdivnochange}
\end{figure}

\begin{lemma}\label{lem:constantdefclasses}
	Let $\Gamma$ be a tropical smooth quartic  curve dual to a triangulation $\mathcal{T}$ of $4\Delta_2$. Let $\BShape$ be a bitangent class of $\Gamma$ with dual bitangent motif belonging to the collection  in Figure \ref{fig:subdivnochange}, modulo  $S_3$-symmetry. Then the shape of the bitangent classes is constant in the deformation class of $\BShape$.
	 \end{lemma}

\begin{proof}
	The proof works similarly for each of the cases. The main argument for each case is that due to the combinatorial structure of $\mathcal{T}$, the two tangencies cannot change the type of their intersection, and this fully determines the bitangent class and its shape. We explain the details for shape (E) and summarize the remaining cases in Table \ref{tab:constdefclasses}.

\begin{figure}[h]
		\centering
			\subfloat{\includegraphics[width=0.15\textwidth]{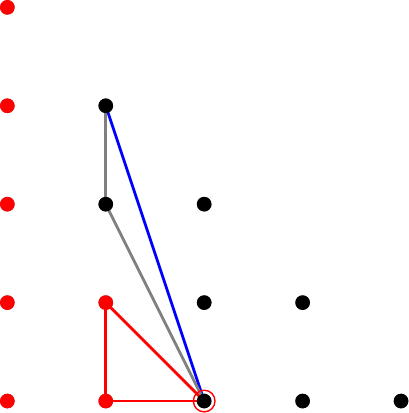}}		 \qquad
			\subfloat{\includegraphics[width=0.3\textwidth]{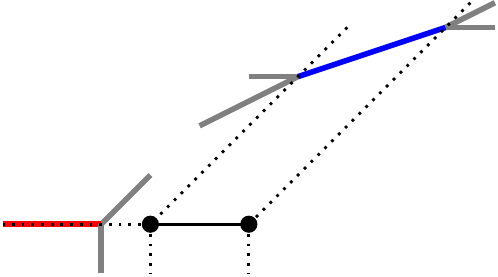}}
		
		\caption{Example of a dual bitangent motif with constant shape (E).}\label{fig:exE}
	\end{figure}

	Let $B$ be a bitangent class of shape (E). 
	We may assume that the dual bitangent motif looks as in Figure \ref{fig:exE}. One tangency point is a transversal intersection in the edge  of $\Gamma$ dual to $E= \overline{p_{20}p_{13}}$. The second tangency point  is a non-transversal intersection of the horizontal ray of $\Lambda$ with the bounded edge of $\Gamma$ dual to  $\overline{p_{10}p_{11}}$ or $\overline{p_{11}p_{12}}$. We can exclude the  edge $\overline{p_{12}p_{13}}$ since it does not define a dual bitangent motif of shape (E). 
	
	The bitangent class  $B$ is a line segment. Its two endpoints are determined by intersecting rays with direction $-(e_1+e_2)$ through the vertices of the edge $E^{\vee}$ of $\Gamma$  with a ray with direction $e_1$ passing through the non-transversal intersection. See Figure \ref{fig:exE} for an example. Independently on the lengths of the edges of $\Gamma$, the endpoints of $B$ cannot lie in $\Gamma$ because of the slopes of the edges of $\Gamma$ that connect the vertices $(\overline{p_{10}p_{11}p_{20}})^\vee$ and $(\overline{p_{12}p_{13}p_{20}})^\vee$.

	Table \ref{tab:constdefclasses} contains the remaining  bitangent classes from the statement. For each deformation class, we draw the dual deformation motif in the dual triangulation and the relevant part of the tropical quartic curve. When there are several possible dual bitangent motifs condensed in a picture, we draw one tropical curve dual to only one of them. From the combinatorial shape of the tropical curve, we can see that changes of edge lengths cannot induce a deformation of the shape of the bitangent class.  
 %\begin{landscape}
 \begin{table}[h] 
 \includegraphics[angle=90,scale=.7]{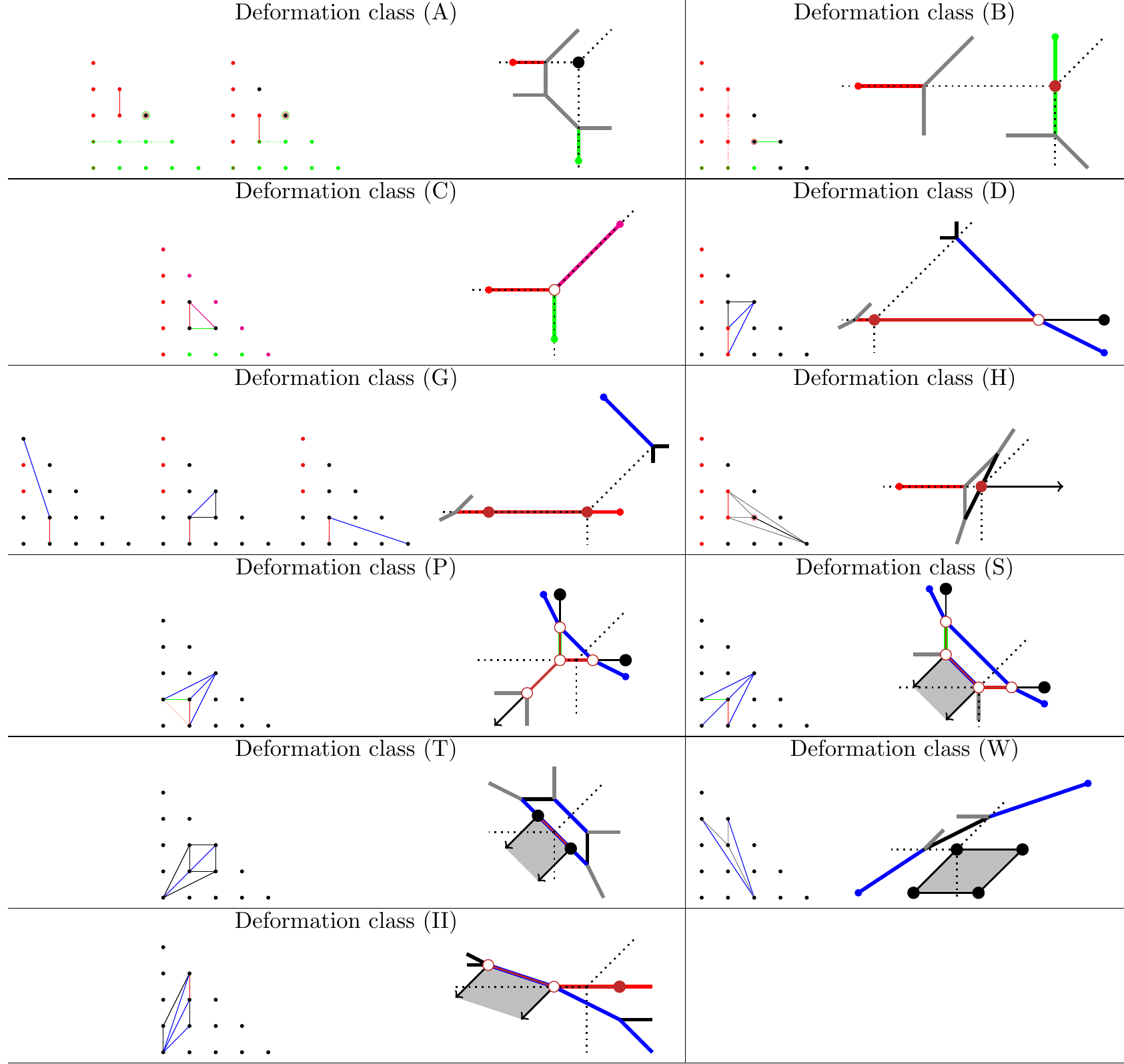}
 \caption{Deformation classes from Lemma \ref{lem:constantdefclasses}. A thickened vertex of an edge in the curve indicates that this edge has to be bounded.} \label{tab:constdefclasses}
 \end{table}
 %\end{landscape}
\end{proof}
 Lemma \ref{lem:constantdefclasses} describes the deformation classes containing  bitangent classes of a unique shape. We now consider two cases of deformation classes of not constant shape. 

\begin{lemma}\label{lem:E-F-J}
	Let $\Gamma$ be a tropical smooth plane quartic curve with dual triangulation $\mathcal{T}$ and $B$ a bitangent class with  dual bitangent motif contained in one of the subcomplexes depicted in Figure \ref{fig:subdiv(E)to(F)to(J)} modulo  $S_3$-symmetry.  For every $c \in \Sigma(\cT)$, the bitangent class $B_c$ can have shapes (E), (F) or (J). 
		\begin{figure}[h]
		\centering
		\includegraphics[width=0.15\textwidth]{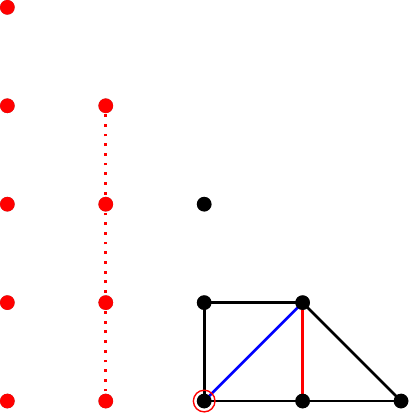}
		\caption{Dual deformation motifs of deformation class (E F J).}
	 \label{fig:subdiv(E)to(F)to(J)}
	\end{figure}
\end{lemma}
\begin{proof}
	Let $B$ be a bitangent class with dual bitangent motif  as in Figure \ref{fig:subdiv(E)to(F)to(J)}. The first tangency point is the transversal intersection of the diagonal ray of the bitangent with the edge of $\Gamma$  dual to $E = \overline{p_{20}p_{31}}$. The second tangency point is a non-transversal intersection with one of the three bounded edges of $\Gamma$ dual to $\overline{p_{10}p_{11}}$, $\overline{p_{11}p_{12}}$ or $\overline{p_{12}p_{13}}$.
		Since  $\mathcal{T}$ is unimodular, we remark that the subcomplexes we are considering contain the dual bitangent motifs of shapes (E), (F) and (J) in identity position. We need to show that for any $\Gamma_c$ with $c\in \Sigma(\mathcal{T})$ the bitagent class $B$ can only deform between these shapes.

	Let $E'$ be the vertical red dotted edge in $\cT$ that forms a triangle with the red circled lattice point $p_{20}$. The bounded edge of $\Gamma$ dual to $E'$ will always have $y$-coordinate smaller than the vertex $v=(\overline{p_{20}p_{31}p_{21}})^\vee$. Hence, for any $\Gamma_c$ the tangent points of the bitangent class are contained in the edge dual to $E$  and the bounded edge dual to $E'$.
	
	The edge length changes  influence the position of the  intersections. Depending on the $y$-coordinate of the bounded edge dual to $E'$ in comparison with the vertex $v$,  we obtain shape (E), (F) or (J) for $B$.   
	This is illustrated in Figure \ref{fig:(E)to(F)to(J)} for the case $E' = \overline{p_{12}p_{13}}$. The other two cases are analogous. We cannot obtain another shape for $B$ because different $x$-coordinates of the points in the edge dual to $E'$ and of $v$  do not influence the shape. 
\begin{figure}[h]
	\captionsetup[subfigure]{labelformat=empty}
	\centering
	\begin{subfigure}{.3\textwidth}
		\centering\includegraphics[width=0.9\textwidth]{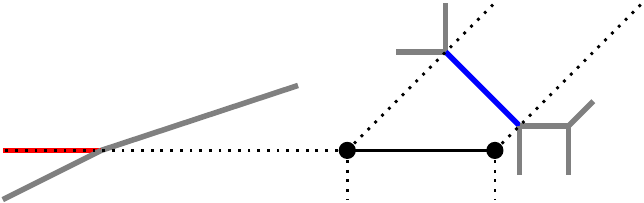}
		\caption{(E)}
	\end{subfigure}
	\begin{subfigure}{.3\textwidth}
		\centering\includegraphics[width=0.9\textwidth]{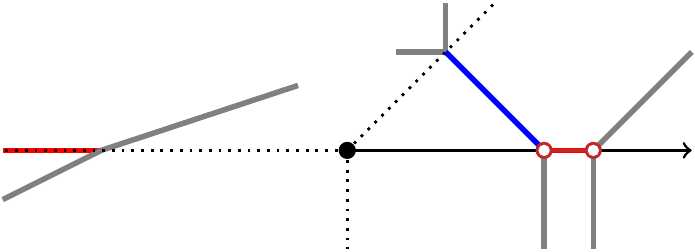}
		\caption{(J)}
	\end{subfigure}
	\begin{subfigure}{.3\textwidth}
		\centering\includegraphics[width=0.9\textwidth]{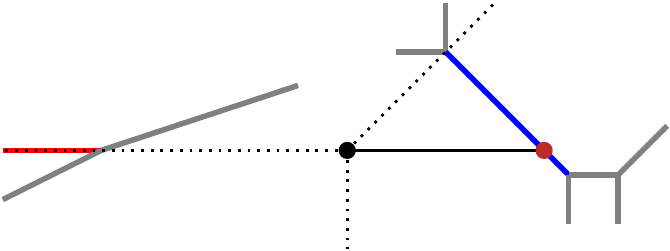}
		\caption{(F)}
	\end{subfigure}
	\caption{Deformation of bitangent shapes (E), (F) and (J).}\label{fig:(E)to(F)to(J)}
\end{figure}

\end{proof}

\begin{lemma}\label{lem:G-K-U-T-V}
	Let $\Gamma$ be a tropical smooth plane quartic curve dual to a triangulation $\mathcal{T}$. Let $B$ be a bitangent class  with  dual bitangent motif contained in one of the two cases illustrated in Figure \ref{fig:subdiv(G)to(K)to(T)to(U)to(V)}, modulo  $S_3$-symmetry. For every $c\in \cT$, the bitangent class $B_c$  can deform through the shapes (G), (K), (U), (U'), (T), (T'), (T''), (V) and into the images of the action by $(x\,y)$  of the shapes (T'), (U), (U'), (K) and (G).
	\begin{figure}[h]
		\centering
		\begin{subfigure}{.2\textwidth}
			\includegraphics[width=0.75\textwidth]{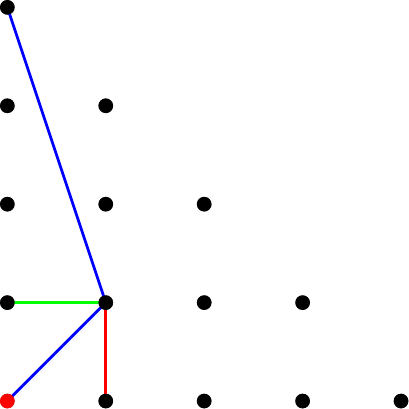}
		\end{subfigure}
		\begin{subfigure}{.2\textwidth}
				\includegraphics[width=0.75\textwidth]{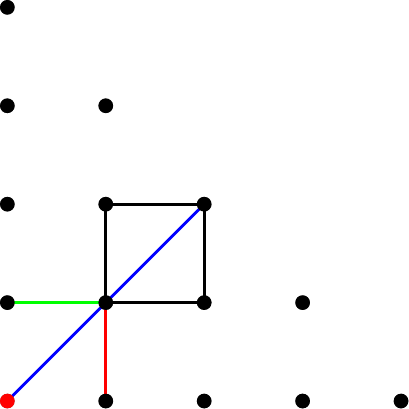}
		\end{subfigure}
		\caption{The two dual deformation motifs of deformation class (G~K~U~U'~T~T'~T''~V)$+(x\,y)$.}\label{fig:subdiv(G)to(K)to(T)to(U)to(V)}
	\end{figure}	
\end{lemma}

\begin{proof}

We need to distinguish two cases in Figure  \ref{fig:subdiv(G)to(K)to(T)to(U)to(V)} depending on whether the edge $\overline{p_{11}p_{22}}$ or the edge $\overline{p_{11}p_{04}}$ is contained in the  triangulation $\cT$.  
	Figure \ref{fig:(G)to(K)to(T)to(U)to(V)} shows deformations of  $B$ into the claimed shapes by edge length changes of $\Gamma$ when the blue edge is $E=\overline{p_{11}p_{22}}$. Analogous pictures can be drawn for $\overline{p_{11}p_{04}}$. It remains to argue that these are the only shapes $B$ can deform into. 
	
	\begin{figure}[h]
		\captionsetup[subfigure]{labelformat=empty}
		\centering
		\begin{subfigure}[b]{.13\textwidth}
			\centering\includegraphics[width=0.9\textwidth]{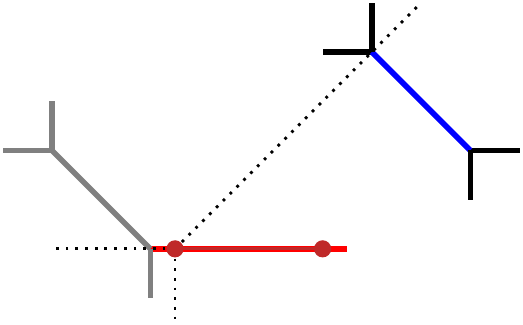}
			\caption{(G)}
		\end{subfigure}
		\begin{subfigure}[b]{.13\textwidth}
			\centering\includegraphics[width=0.9\textwidth]{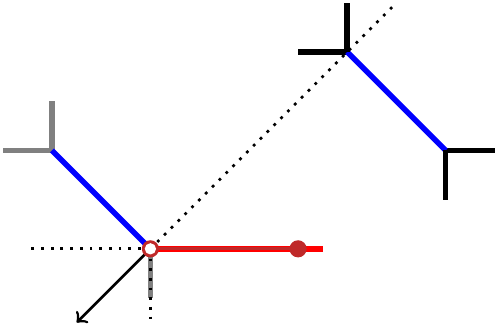}
			\caption{(K)}
		\end{subfigure}
		\begin{subfigure}[b]{.13\textwidth}
			\centering\includegraphics[width=0.9\textwidth]{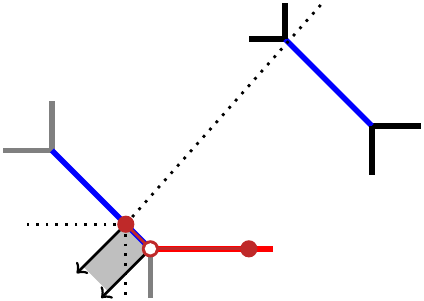}
			\caption{(U)}
		\end{subfigure}
		\begin{subfigure}[b]{.13\textwidth}
			\centering\includegraphics[width=0.9\textwidth]{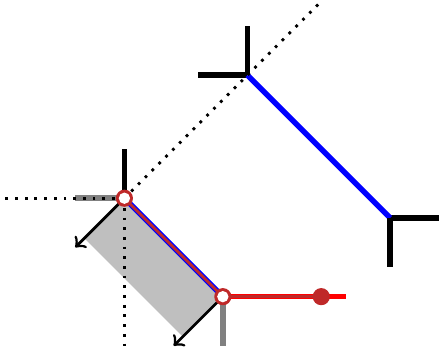}
			\caption{(U')}
		\end{subfigure}
		\begin{subfigure}[b]{.13\textwidth}
			\centering\includegraphics[width=0.9\textwidth]{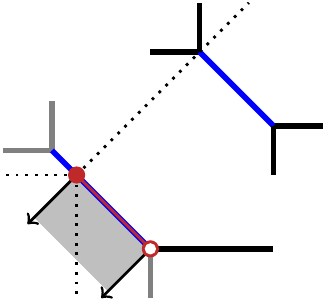}
			\caption{(T')}
		\end{subfigure}
		\begin{subfigure}[b]{.13\textwidth}
			\centering\includegraphics[width=0.9\textwidth]{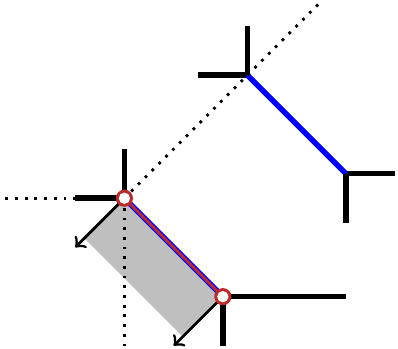}
			\caption{(T'')}
		\end{subfigure}
		\begin{subfigure}[b]{.13\textwidth}
			\centering\includegraphics[width=0.9\textwidth]{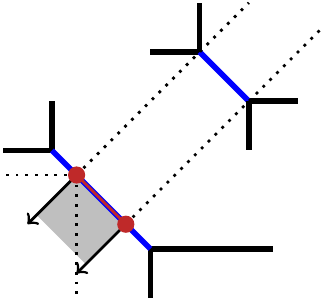}
			\caption{(T)}
		\end{subfigure}
	\begin{subfigure}[b]{.14\textwidth}
		\centering\includegraphics[width=0.9\textwidth]{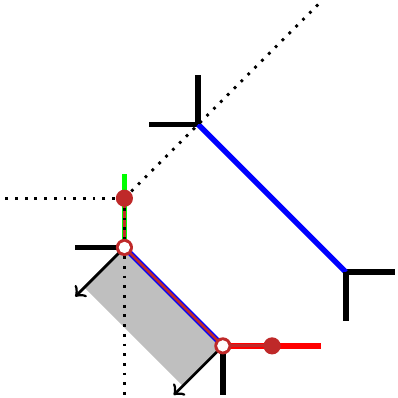}
		\caption{(V)}
	\end{subfigure}
		\begin{subfigure}[b]{.14\textwidth}
			\centering\includegraphics[width=0.9\textwidth]{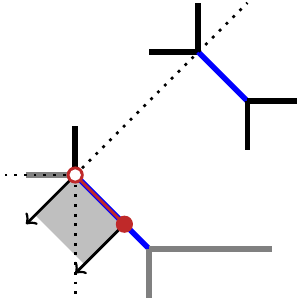}
			\caption{(T')$_{(xy)}$ }
		\end{subfigure}
		\begin{subfigure}[b]{.14\textwidth}
			\centering\includegraphics[width=0.9\textwidth]{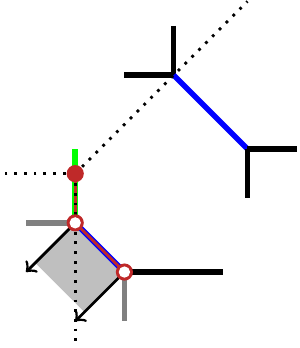}
			\caption{(U')$_{(xy)}$ }
		\end{subfigure}
		\begin{subfigure}[b]{.14\textwidth}
			\centering\includegraphics[width=0.9\textwidth]{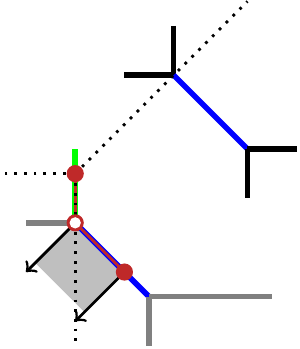}
			\caption{(U)$_{(xy)}$ }
		\end{subfigure}
		\begin{subfigure}[b]{.14\textwidth}
			\centering\includegraphics[width=0.9\textwidth]{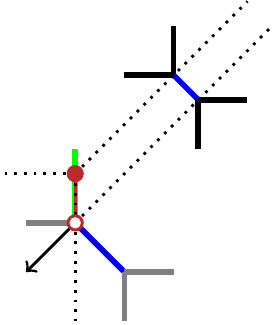}
			\caption{(K)$_{(xy)}$ }
		\end{subfigure}
		\begin{subfigure}[b]{.14\textwidth}
			\centering\includegraphics[width=0.9\textwidth]{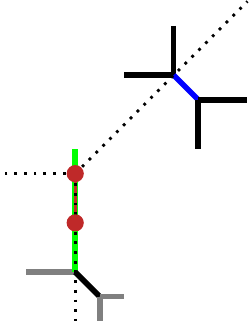}
			\caption{(G)$_{(xy)}$ }
		\end{subfigure}
		\caption{Deformation of bitangent shapes (G), (K), (U), (U'), (T''), (T), (V), (T'), (U')$_{(x\,y)}$, (U)$_{(x\,y)}$, (K)$_{(x\,y)}$, (G)$_{(x\,y)}$.}\label{fig:(G)to(K)to(T)to(U)to(V)}
	\end{figure}

	We see that $B$ cannot deform into any other shapes by considering Figure \ref{fig:(G)to(K)to(T)to(U)to(V)}. Rays with direction $e_1+e_2$ passing through the vertices of $E^{\vee}$ cannot  intersect the upper vertex of $(\overline{p_{01}p_{11}})^\vee$ or the right vertex of $(\overline{p_{10}p_{11}})^{\vee}$. Thus, $B$ cannot deform further. 	
\end{proof}

Similar reasoning leads us to complete the classification in Theorem \ref{thm:classification}. Orbit representatives of the dual complexes describing the deformation classes are illustrated in Figure \ref{fig:defclasses}.\\

\emph{Proof of Theorem \ref{thm:classification}.} 
	\begin{longtable}{m{5cm}|m{8.5cm}}		
		\hline\multicolumn{2}{c}{\textbf{(B H' H)}}\\
		\hline 
			\includegraphics[width = 0.15\textwidth]{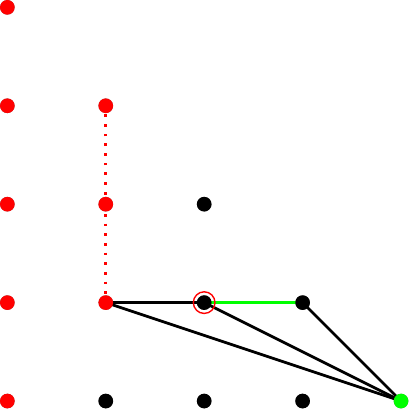} & 
First tangency: Non-transversal intersection on  $(\overline{p_{11}p_{12}})^\vee$ or $(\overline{p_{12}p_{13}})^\vee$.

Second tangency:  Depending on the edge lengths, vertex of the bitangent lying on the edge $(\overline{p_{21}p_{31}})^\vee$ (B), $(\overline{p_{21}p_{40}})^\vee$ (H) or on their shared vertex (H').  

No further deformations: The vertices $\overline{p_{40}p_{21}p_{11}}^\vee$ resp. $\overline{p_{21}p_{31}p_{40}}^\vee$ must have $y$-coordinate smaller resp. larger than $(\overline{p_{11}p_{12}})^\vee$ and $(\overline{p_{12}p_{13}})^\vee$. 
 \\
		\hline
		
			\newpage 
		\hline
		\multicolumn{2}{c}{\textbf{(B H' H)$+(y\,z)$}}\\ \hline
		\includegraphics[width = 0.15\textwidth]{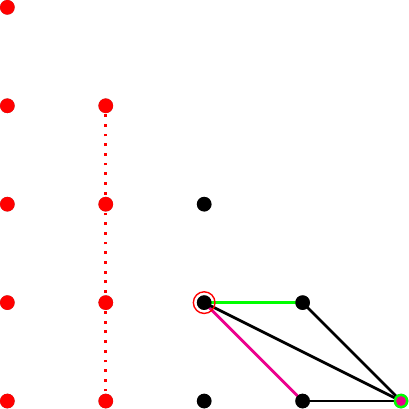} & 
		First tangency: Non-transversal intersection 
		on $(\overline{p_{1v}p_{1,v+1}})^\vee$.

Second tangency:   Depending on the edge lengths, vertex of the bitangent lying on the edge $(\overline{p_{21}p_{31}})^\vee$ (B), $(\overline{p_{21}p_{40}})^\vee$ (H),  $(\overline{p_{21}p_{30}})^\vee$  (B)$_{(yz)}$ or on their shared vertices (H') resp. (H')$_{(yz)}$.

No further deformations: The lowest vertex of $(\overline{p_{21}p_{30}})^\vee$ has $y$-coordinate smaller than $(\overline{p_{21}p_{1v}p_{1,v+1}})^\vee$.
\\
		\hline
		\multicolumn{2}{c}{\textbf{(B M)$+(y\,z)$ }}\\ \hline
		\includegraphics[width = 0.15\textwidth]{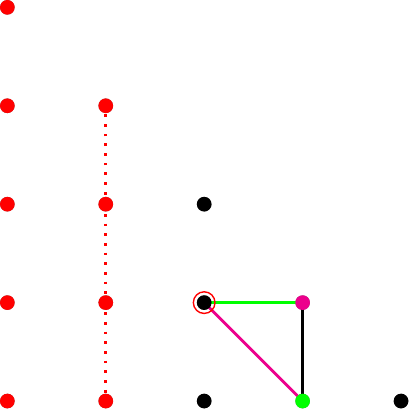} &  
			First tangency: Non-transversal intersection on  $(\overline{p_{1v}p_{1,v+1}})^\vee$.
				
		Second tangency: Depending on the edge lengths, vertex of the bitangent lying on the edge $(\overline{p_{21}p_{31}})^\vee$ (B),   $(\overline{p_{21}p_{30}})^\vee$  (B)$_{(yz)}$ or on the ray of direction $e_1$ starting at $(\overline{p_{21}p_{30}p_{31}})^\vee$ (M).
		
		No further deformations: The lower vertex of $(\overline{p_{21}p_{30}})^\vee$ has always smaller $y$-coordinate than $(\overline{p_{1v}p_{1,v+1}p_{21}})^\vee$ and the upper vertex of $(\overline{p_{21}p_{31}})^\vee$ has always larger $y$-coordinate than $(\overline{p_{1v}p_{1,v+1}})^\vee$.
	\\
		\hline
		\multicolumn{2}{c}{\textbf{(D L' Q) }}\\ \hline 
		\includegraphics[width = 0.15\textwidth]{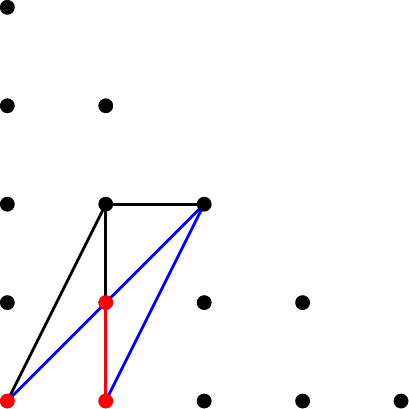} &
		First tangency: Transversal intersection of the diagonal ray resp. vertical ray of the bitangent line with $(\overline{p_{11}p_{22}})^\vee$ resp. $(\overline{p_{22}p_{10}} )^\vee$.
		
		Second tangency: Depending on the edge lengths, the non-transversal intersection with $(\overline{p_{10}p_{11}})^\vee$  
		in the horizontal ray of the bitangent line (D) deforms to a transversal intersection of  $(\overline{p_{00}p_{11}})^\vee$ with the diagonal ray 
		(Q), through the shared vertex (L').
	
	No further deformation: Due to the edge directions determined by the dual bitangent motif.
		\\ \hline
	\newpage 
		\hline
		\multicolumn{2}{c}{\textbf{(D L' Q Q' R)}} \\ \hline\includegraphics[width = 0.15\textwidth]{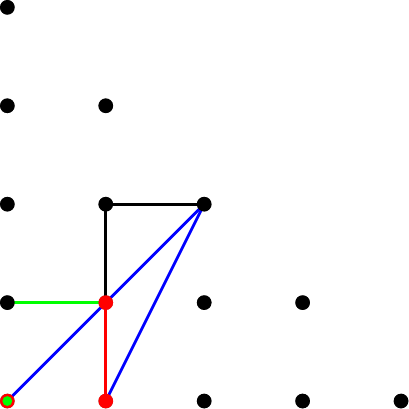} &
		
			First tangency: Same as for (D L' Q).
		
		Second tangency: Depending on the edge lengths, the non-transversal intersection with $(\overline{p_{10}p_{11}})^\vee$ in the horizontal ray of the bitangent line (D) deforms through the shared vertex (L') to a transversal intersection of $(\overline{p_{00}p_{11}})^\vee$ with the diagonal ray (Q), through the shared vertex (Q') to a non-transversal intersection of $(\overline{p_{01}p_{11}})^\vee$ with the vertical ray (R).
		
		No further deformation: The upper vertex of $(\overline{p_{01}p_{11}})^\vee$ has always larger $y$-coordinate than any intersection of the ray of direction $-(e_1+e_2)$ starting at $(\overline{p_{22}p_{11}p_{12}})^\vee$ with $(\overline{p_{01}p_{11}})^\vee$.
		 \\ \hline
		\multicolumn{2}{c}{\textbf{(D L O)}}\\ \hline
		\includegraphics[width = 0.15\textwidth]{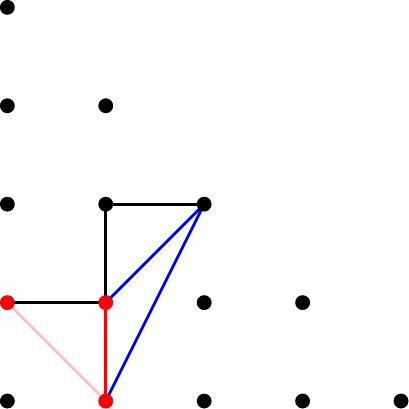} & 
		
		 	First tangency: Same as for (D L' Q).
		 
		 Second tangency: Depending on the edge lengths, the non-transversal intersection with $(\overline{p_{10}p_{11}})^\vee$ in the horizontal ray of the bitangent line (D) deforms through the shared vertex (L) to a non-transversal intersection of $(\overline{p_{01}p_{11}})^\vee$ with vertical ray (O).
		 
		 No further deformation: Same as for (D-L'-Q-Q'-R).
		\\ \hline
	
		\multicolumn{2}{c}{\textbf{(G I N)$+(x\,y)$}}\\ \hline
		
		\includegraphics[width = 0.15\textwidth]{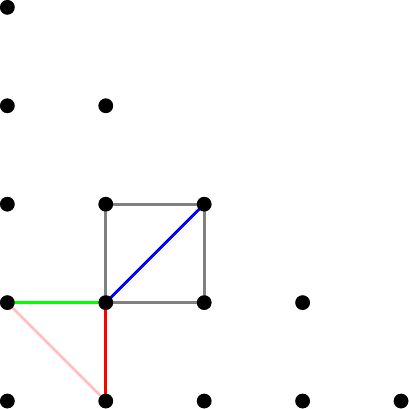}  \includegraphics[width = 0.15\textwidth]{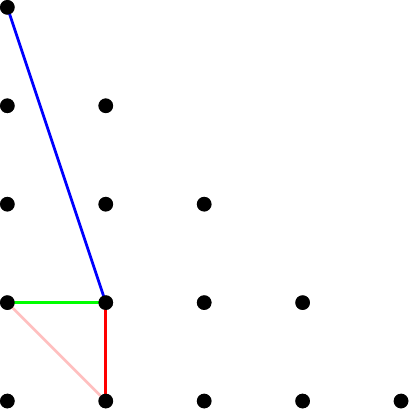} & 
		First tangency: Transversal intersection of the diagonal ray with $(\overline{p_{11}p_{22}})^\vee$ resp. $(\overline{p_{11}p_{04}})^\vee$.
		
		Second tangency:   Non-transversal intersection of the horizontal ray with $(\overline{p_{10}p_{11}})^\vee$ (G,I,N), non-transversal intersection of the diagonal ray  with $(\overline{p_{10}p_{01}})^\vee$ (I,N), non-transversal intersection of the horizontal ray  with $(\overline{p_{01}p_{11}})^\vee$ (N, I$_{(xy)}$,G$_{(xy)}$), depending on the edge lengths.
		
		No further deformations: Rays of direction $-(e_1+e_2)$ starting from the two vertices of $(\overline{p_{11}p_{22}})^\vee$ resp. $(\overline{p_{11}p_{04}})^\vee$ always intersect $(\overline{p_{01}p_{11}})^\vee$ resp.  $(\overline{p_{10}p_{11}})^\vee$ under resp. left of their other vertex (\textbf{not} $(\overline{p_{10}p_{01}p_{11}})^\vee$).
		\\ \hline
			\newpage 
		\hline
		\multicolumn{2}{c}{\textbf{(G K U T T')}}\\ \hline
		\includegraphics[width = 0.15\textwidth]{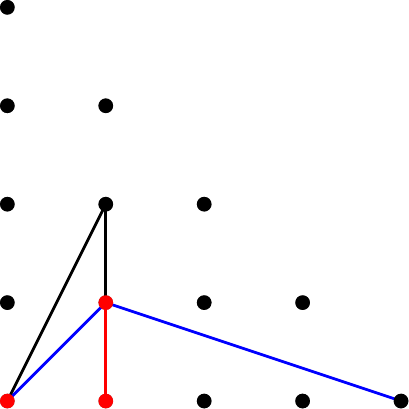} \includegraphics[width = 0.15\textwidth]{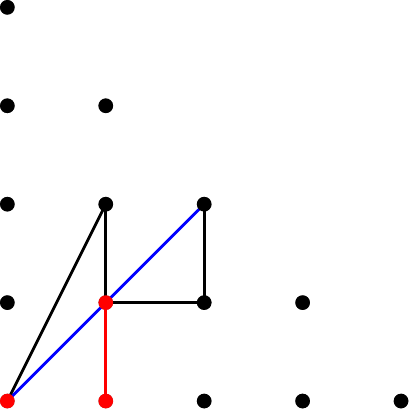} & 
		First tangency: Transversal intersection of the diagonal ray with $(\overline{p_{11}p_{22}})^\vee$ resp. $(\overline{p_{11}p_{40}})^\vee$.
		
		Second tangency:  Non-transversal intersection of the horizontal ray with $(\overline{p_{10}p_{11}})^\vee$ (G,K,U), transversal intersection of the diagonal ray with $(\overline{p_{10}p_{01}p_{00}})^\vee$ (K,U,T'),  transversal intersection of the diagonal ray with $(\overline{p_{10}p_{01}})^\vee$ (U,T',T), depending on the edge lengths.
		
		No further deformations: Similar to (G I N)$+(x\,y)$. \\
			\hline 

					\multicolumn{2}{c}{\textbf{(W X Y EE GG )}}\\ \hline
					\includegraphics[width = 0.15\textwidth]{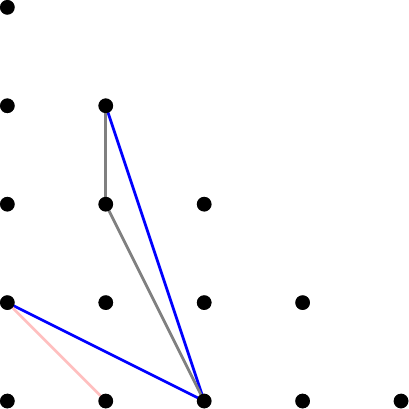} & First tangency: Transversal intersection of the diagonal ray with $(\overline{p_{20}p_{13}})^\vee$.
					
					Second tangency: Transversal intersection of the horizontal ray with $(\overline{p_{01}p_{20}})^\vee$ (W). Depending on the edge lengths, additionally: Transversal intersection with $(\overline{p_{10}p_{01}})^\vee$ (X$_{(xz)}$,Y$_{(xz)}$,GG), vertex of bitangent line contained in $(\overline{p_{01}p_{20}})^\vee$ (Y$_{(xz)}$,EE,GG).
					
					No further deformations: Due to its position and slope, $(\overline{p_{20}p_{13}})^\vee$ never intersects the the bitangent class; the diagonal ray starting from  $(\overline{p_{12}p_{20}p_{13}})^\vee$ never meets  $(\overline{p_{01}p_{20}p_{11}})^\vee$; the intersection of the  diagonal rays from the vertices of $(\overline{p_{20}p_{13}})^\vee$ with the horizontal rays from the vertices of $(\overline{p_{01}p_{20}})^\vee$ always lead to a $2$-dimensional bounded cell.	\\ \hline
					\multicolumn{2}{c}{\textbf{(W...HH)$+(x\, z)$}}\\ \hline
					\includegraphics[width = 0.15\textwidth]{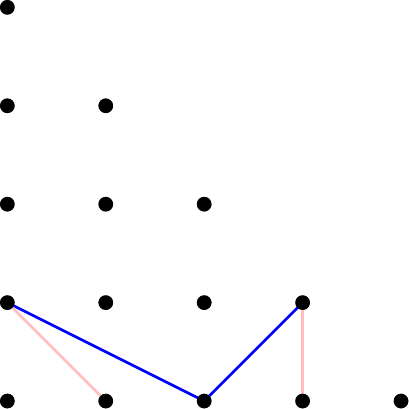} &
					Due to $(x\, z)$-symmetry, we describe the tangency types only once.
							
					Tangency: Transversal intersection of the diagonal ray of the bitangent line with $(\overline{p_{20}p_{31}})^\vee$.
					Depending on the edge lengths, there are additional types: transversal intersection with $(\overline{p_{30}p_{31}})^\vee$  \begin{small}(X,Y,Z,AA,AA$_{(xz)}$,BB,CC,DD,FF,GG$_{(xz)}$,HH,HH$_{(xz)}$), \end{small}% resp. $(\overline{p_{10}p_{01}})^\vee$ (X$_{(xz)}$,Y$_{(xz)}$,Z,AA,AA$_{(xz)}$,BB,CC$_{(xz)}$, DD,FF$_{(xz)}$,GG,HH,HH$_{(xz)}$)
					vertex of bitangent line contained in $(\overline{p_{31}p_{20}})^\vee$ \begin{small}(Y,AA,BB,CC,CC$_{(xz)}$,DD,EE$_{(xz)}$,FF$_{(xz)}$,GG$_{(xz)}$,HH$_{(xz)}$).\end{small}% resp. $(\overline{p_{01}p_{20}})^\vee$ (Y$_{(xz)}$,AA$_{(xz)}$,BB,CC,CC$_{(xz)}$,DD,EE,FF,GG,HH).		
							
					No further deformations: $(\overline{p_{20}p_{13}})^\vee$ and $(\overline{p_{01}p_{20}})^\vee$ are always in relative position such that the intersection of rays of direction $-(e_1+e_2)$ starting from the vertices of $(\overline{p_{20}p_{13}})^\vee$ with rays of direction $e_1$ starting from the vertices of $(\overline{p_{01}p_{20}})^\vee$ (with $\Gamma$ as boundary where necessary) will always form a 2-dimensional cell in the area between the two edges.	 \\ \hline
			\end{longtable}
\hfill $\Box$

\clearpage
\begin{landscape}
	\begin{figure}
		\centering
		\begin{subfigure}{.35\textwidth}
			\centering
			\begin{subfigure}{.45\textwidth}
				\centering
				\includegraphics[width=0.9\textwidth]{figures/defclass/_A__1.pdf}
			\end{subfigure}	\begin{subfigure}{.45\textwidth}
				\centering
				\includegraphics[width=0.9\textwidth]{figures/defclass/_A__2.pdf}
			\end{subfigure}
			\caption*{(A)}
		\end{subfigure}
		\begin{subfigure}{.15\textwidth}
			\centering
			\includegraphics[width=0.9\textwidth]{figures/defclass/_B_-_H__-_H_.pdf}
			\caption*{(B H' H)}
		\end{subfigure}
		\begin{subfigure}{.15\textwidth}
			\centering
			\includegraphics[width=0.9\textwidth]{figures/defclass/_B_-_H__-_H_-_H__-_B_.pdf}
			\caption*{(B H' H)+$(yz)$}
		\end{subfigure}
		\begin{subfigure}{.15\textwidth}
			\centering
			\includegraphics[width=0.9\textwidth]{figures/defclass/_B_-_M_-_B_.pdf}
			\caption*{(B M)+$(yz)$}
		\end{subfigure}
		\begin{subfigure}{.15\textwidth}
			\centering
			\includegraphics[width=0.9\textwidth]{figures/defclass/_B_.pdf}
			\caption*{(B)}
		\end{subfigure}
		\begin{subfigure}{.15\textwidth}
			\centering
			\includegraphics[width=0.9\textwidth]{figures/defclass/_C_.pdf}
			\caption*{(C)}
		\end{subfigure}
		\begin{subfigure}{.15\textwidth}
			\centering
			\includegraphics[width=0.9\textwidth]{figures/defclass/_D_-_L__-_Q_.pdf}
			\caption*{(D L' Q)}
		\end{subfigure}
		\begin{subfigure}{.15\textwidth}
			\centering
			\includegraphics[width=0.9\textwidth]{figures/defclass/_D_-_L__-_Q_-_Q__-_R_.pdf}
			\caption*{(D L' Q Q' R)}
		\end{subfigure}
		\begin{subfigure}{.15\textwidth}
			\centering
			\includegraphics[width=0.9\textwidth]{figures/defclass/_D_-_L_-_O_.pdf}
			\caption*{(D L O)}
		\end{subfigure}
		\begin{subfigure}{.15\textwidth}
			\centering
			\includegraphics[width=0.7\textwidth]{figures/subdiv_D_.pdf}
			\caption*{(D)}
		\end{subfigure}
		\begin{subfigure}{.15\textwidth}
			\centering
			\includegraphics[width=0.9\textwidth]{figures/defclass/_E_.pdf}
			\caption*{(E)}
		\end{subfigure}
		\begin{subfigure}{.15\textwidth}
			\centering
			\includegraphics[width=0.9\textwidth]{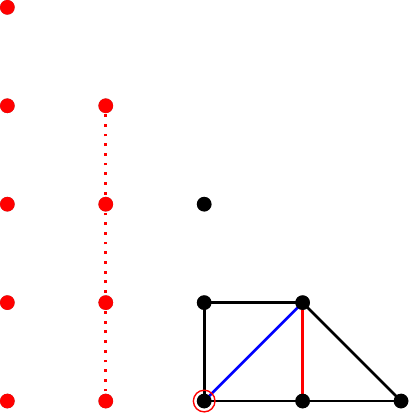}
			\caption*{(E F J)}
		\end{subfigure}
		\begin{subfigure}{.35\textwidth}
			\centering
			\begin{subfigure}{.45\textwidth}
				\centering
				\includegraphics[width=0.9\textwidth]{figures/defclass/_G_-_I_-_N_-_I_-_G_1.pdf}
			\end{subfigure}	\begin{subfigure}{.45\textwidth}
				\centering
				\includegraphics[width=0.9\textwidth]{figures/defclass/_G_-_I_-_N_-_I_-_G_2.pdf}
			\end{subfigure}
			\caption*{(G I N) +$(xy)$}
		\end{subfigure}
		\begin{subfigure}{.35\textwidth}
			\centering
			\begin{subfigure}{.45\textwidth}
				\centering
				\includegraphics[width=0.9\textwidth]{figures/defclass/_G_-_K_-_U_-_T_-_T__1.pdf}
			\end{subfigure}	\begin{subfigure}{.45\textwidth}
				\centering
				\includegraphics[width=0.9\textwidth]{figures/defclass/_G_-_K_-_U_-_T_-_T__2.pdf}
			\end{subfigure}
			\caption*{(G K U T T')}
		\end{subfigure}
		\begin{subfigure}{.35\textwidth}
			\begin{subfigure}{.45\textwidth}
				\centering
				\includegraphics[width=0.9\textwidth]{figures/defclass/_G_-_K_-_U_-_T_-_T__-_T___-_V_1.pdf}
			\end{subfigure}
			\begin{subfigure}{.45\textwidth}
				\centering
				\includegraphics[width=0.9\textwidth]{figures/defclass/_G_-_K_-_U_-_T_-_T__-_T___-_V_2.pdf}
			\end{subfigure}
			\caption*{(G K U U' T T' T'' V)+$(x\,y)$}
		\end{subfigure}
		\begin{subfigure}{.5\textwidth}
			\centering
			\begin{subfigure}{.3\textwidth}
				\centering
				\includegraphics[width=0.9\textwidth]{figures/defclass/_G_1.pdf}
			\end{subfigure}
			\begin{subfigure}{.3\textwidth}
				\centering
				\includegraphics[width=0.9\textwidth]{figures/defclass/_G_2.pdf}
			\end{subfigure}
			\begin{subfigure}{.3\textwidth}
				\centering
				\includegraphics[width=0.9\textwidth]{figures/defclass/_G_3.pdf}
			\end{subfigure}	
			\caption*{(G)}
		\end{subfigure}
		\begin{subfigure}{.15\textwidth}
			\centering
			\includegraphics[width=0.9\textwidth]{figures/defclass/_H_.pdf}
			\caption*{(H)}
		\end{subfigure}
		\begin{subfigure}{.15\textwidth}
			\centering
			\includegraphics[width=0.9\textwidth]{figures/defclass/_P_.pdf}
			\caption*{(P)}
		\end{subfigure}
		\begin{subfigure}{.15\textwidth}
			\centering
			\includegraphics[width=0.9\textwidth]{figures/defclass/_S_.pdf}
			\caption*{(S)}
		\end{subfigure}
		\begin{subfigure}{.15\textwidth}
			\centering
			\includegraphics[width=0.9\textwidth]{figures/defclass/_T_.pdf}
			\caption*{(T)}
		\end{subfigure}
		\begin{subfigure}{.15\textwidth}
			\centering
			\includegraphics[width=0.9\textwidth]{figures/defclass/_W_.pdf}
			\caption*{(W)}
		\end{subfigure}
		\begin{subfigure}{.2\textwidth}
			\centering
			\includegraphics[width=0.73\textwidth]{figures/defclass/_W_-_X_-_Y_-_EE_-_GG_.pdf}
			\caption*{(W X Y EE GG)}
		\end{subfigure}
		\begin{subfigure}{.15\textwidth}
			\centering
			\includegraphics[width=0.9\textwidth]{figures/defclass/_W_-to-_HH_.pdf}
			\caption*{(W...HH)+$(xz)$}
		\end{subfigure}
		\begin{subfigure}{.15\textwidth}
			\centering
			\includegraphics[width=0.9\textwidth]{figures/defclass/_II_.pdf}
			\caption*{(II)}
		\end{subfigure}
		\caption{A list of the dual deformation motifs of all 24 deformation classes.}\label{fig:defclasses}
	\end{figure}
	
\end{landscape}

\section{Real lifting conditions of deformation classes}\label{sec:lifting}
\noindent
In this section, we focus on the real lifting conditions determined by each shape in a  deformation class. This leads to a new proof of Pl\"ucker and Zeuthen's count for tropically smooth quartics.  
As before, we study few cases in detail and give an overview for the remaining ones. The deformation class (C) requires a special argument, as is explained in Example \ref{exa:C}. 

\begin{lemma}\label{prop:signsEFJ}
	Let $\Gamma$ be a smooth tropical plane quartic with dual triangulation $\cT$ and a bitangent shape $B$ in the deformation class (E F J). For every $c\in\Sigma(\cT)$ the real lifting conditions of $B_c$ in $\Gamma_c$ are independent of the shape of the bitangent class. 
\end{lemma}

\begin{proof} From the proof of Lemma \ref{lem:E-F-J} we know that a bitangent class in the deformation class (E F J) can deform into any of the three shapes, all in the same position with respect to the action of $S_3$. % in their identity positions. 
	In order to find the real lifting conditions, we only need to consult Table \ref{tab:tab11} containing the real lifting conditions for the identity positions. As the triangulation $\cT$ is fixed, the values that have to be substituted for $v$ and $i$ in the formula from Table \ref{tab:tab11} do not change for the three shapes. 	
\end{proof}

\begin{lemma}\label{prop:signsGKUTV}
	Let $\Gamma$ be a smooth tropical plane quartic curve with dual triangulation $\cT$ and a bitangent shape $B$ in a deformation class (G K U U' T T' T'' V)+$(x\,y)$. For every $c\in\Sigma(\cT)$ the real lifting conditions of $B_c$ in $\Gamma_c$ are independent of the shape of the bitangent class. 
\end{lemma}

\begin{proof} Figure \ref{fig:subdiv(G)to(K)to(T)to(U)to(V)} shows the two dual deformation motifs of the deformation class. If the dual deformation motif contains the edge $\overline{p_{11}p_{04}}$, we obtain the following list of lifting conditions from Table \ref{tab:tab11} for the bitangent shapes in the deformation class:
	\medskip 
	
	\begin{tabular}{cccc}
		shape  & permutation & lifting condition& parameters\\
		\hline
		(G)  &identity &  $(-s_{10}s_{11})^{i}s_{0i}s_{04}>0$ & $i=$ $y$-coor. of vertex at $x=0$ \\ &&  & forming a triangle with $\overline{p_{10}p_{11}}$\\
		&&  & $p_{04}$ vertex of edge $\overline{p_{11}p_{k,4-k}}$\\% vertex of the blue edge\\
		\hline
		(K),(T),(U),(V)  &identity & $ s_{00}s_{04}>0$ & $p_{04}$ vertex of edge $\overline{p_{11}p_{k,4-k}}$ \\
		\hline
		(K),(T),(U),(V)  &$(x\,y)$ & $ s_{00}s_{04}>0$ & $p_{04}$ vertex of edge $\overline{p_{11}p_{4-k,k}}$ \\
		\hline
		(G)  &$(x\,y)$ &  $(-s_{01}s_{11})^{i}s_{i0}s_{04}>0$ & $i=$ $x$-coor. of vertex at $y=0$\\
		 &&  &  forming a triangle with $\overline{p_{01}p_{11}}$,\\
		&&  & $p_{04}$ vertex of $\overline{p_{11}p_{4-k,k}}$\\
		
	\end{tabular}

\medskip 

	The value of $k$ changes when we consider the $(x\,y)$ permutation. However, the (blue) edge $\overline{p_{11}p_{04}}$ in the subdivision stays the same in all cases. Since the vertex relevant to the value of $i$ is $p_{00}$ in both cases, we substitute $i=0$ and obtain the real lifting condition 	$s_{00}s_{04}>0$ 	for all shapes in this deformation class.
	
 If the dual deformation motif contains the edge $\overline{p_{11}p_{22}}$, the situation is analogous:

\medskip 
\begin{tabular}{cccc}
	shape  & permutation & lifting condition& parameters\\
	\hline
	(G)  &identity &  $(-s_{10}s_{11})^{i}s_{0i}s_{22}>0$ & $i=$ $y$-coor. of vertex at $x=0$ \\ &&  & forming a triangle with $\overline{p_{10}p_{11}}$,\\
	&&  &  $p_{22}$ vertex of edge $\overline{p_{11}p_{k,4-k}}$\\% vertex of the blue edge\\
	\hline
	(K),(T),(U),(V)  &identity & $ s_{00}s_{22}>0$ & $p_{22}$ vertex of edge $\overline{p_{11}p_{k,4-k}}$ \\
	\hline
	(K),(T),(U),(V)  &$(x\,y)$ & $ s_{00}s_{22}>0$ & $p_{22}$ vertex of edge $\overline{p_{11}p_{4-k,k}}$ \\
	\hline
	(G)  &$(x\,y)$ &  $(-s_{01}s_{11})^{i}s_{i0}s_{22}>0$ & $i=$ $x$-coor. of vertex at $y=0$\\
	&&  & forming a  triangle with $\overline{p_{01}p_{11}}$,\\
	&&  & $p_{22}$ vertex of $\overline{p_{11}p_{4-k,k}}$\\
	
\end{tabular}
\medskip 

As before, $i=0$ and the real lifting condition is $s_{00}s_{22}>0$ for all bitangent shapes in the deformation class. Thus, we can conclude for both dual deformation motifs in Figure~\ref{fig:subdiv(G)to(K)to(T)to(U)to(V)}  that the real lifting conditions are independent of the shapes. 
\end{proof}

We now focus  on the special case of deformation class (C). In \cite{CueMa20}, the lifting conditions for bitangent class (C) are computed for generic tropical quartics satisfying the following condition: 
If $\Gamma$ contains a vertex $v$ adjacent to three bounded edges with directions $-e_1$, $-e_2$
and $e_1 + e_2$, then there exists a unique shortest edge. 
 The vertex of a tropical bitangent of shape (C) coincides with $v$. 
Cueto and Marking chose the edge with direction $-e_2$ as shortest edge. Any generic tropical quartic having a bitangent class of shape (C) at a vertex $v$, but with different edge lengths, can be brought into this position by applying an action of $S_3$. This  changes the dual subdivision accordingly and, as consequence, also the formula for the real lifting conditions of (C). 
We illustrate this in the following example. 

\begin{exa}\label{exa:C}
	We consider the two smooth tropical quartic curves dual to the triangulation $\cT$  shown in Figure \ref{fig:exC}.  These quartic curves have a bitangent class of shape (C). We denote with $\lambda_1$, $\lambda_2$ and $\lambda_3$ the lattice lengths of the edges adjacent to the vertex that forms the bitangent class of shape (C) with direction $-e_2$ , $-e_1$ and $e_1+e_2$, respectively. For the tropical curve in Figure \ref{fig:exaC1}, these lengths satisfy what we call the \emph{identity case} of the genericity condition: $\lambda_1<\lambda_2\leq \lambda_3$.  We substitute $i=2,\, j=1, \, k=2$ in the real lifting condition for (C) in Table \ref{tab:tab11} obtaining
\begin{equation} \label{eq:gencase}
	-s_{11}s_{21}s_{02}s_{10}>0 \text{ and } -s_{21}s_{11}s_{22}s_{10}>0.
\end{equation}

	\begin{figure}[h]
		\centering
		\begin{subfigure}[b]{.20\textwidth}
			\centering	\includegraphics[width=0.9\textwidth]{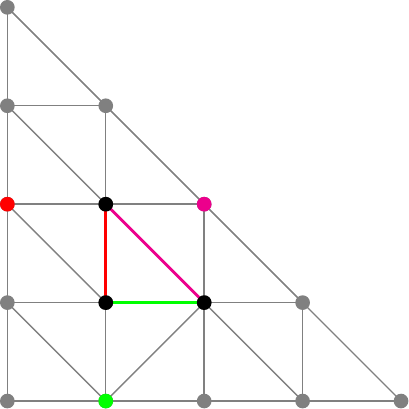}
			\caption{Triangulation $\cT$}\label{fig:exaC1subdiv}
		\end{subfigure}\quad 
		\begin{subfigure}[b]{.25\textwidth}
			\centering
			\includegraphics[width=0.9\textwidth]{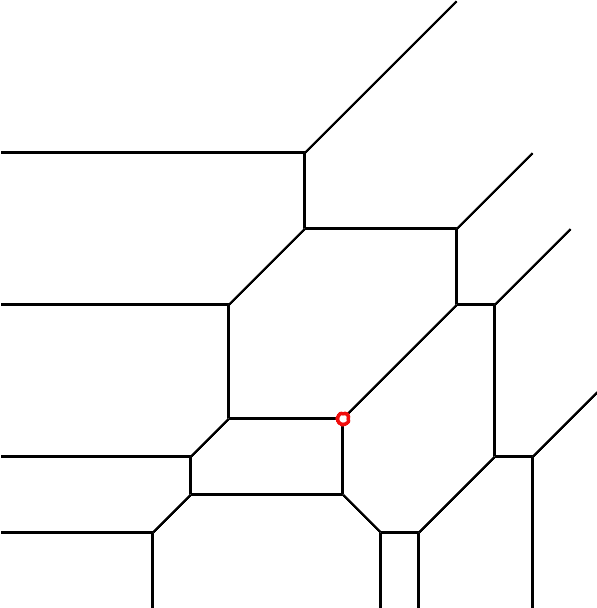}
			\caption{Generic curve in identity position.}\label{fig:exaC1}
		\end{subfigure} \quad
	\begin{subfigure}[b]{.25\textwidth}
		\centering
		\includegraphics[width=0.9\textwidth]{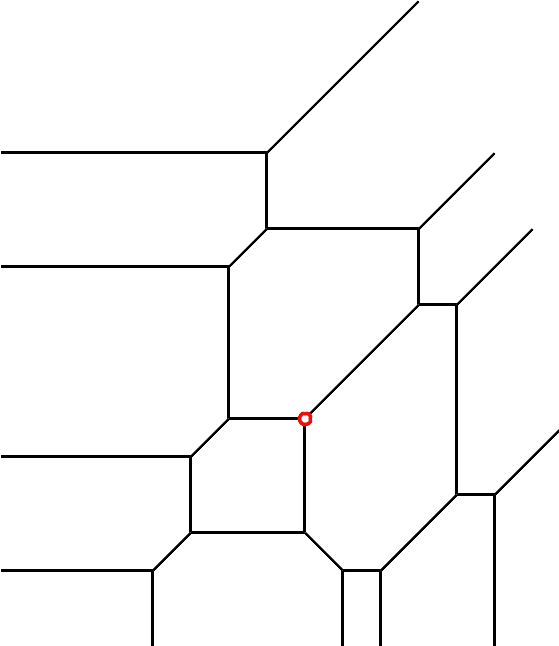}
		\caption{Generic curve not in identity position.}\label{fig:exaC2}
	\end{subfigure}
		\caption{ Smooth tropical quartic with bitangent class of shape (C).}\label{fig:exC}
	\end{figure}
By choosing a different weight vector in $\Sigma(\cT)$, we can deform the edge lengths such that $\lambda_2<\lambda_1\leq \lambda_3$. An example of this is shown in Figure \ref{fig:exaC2}. In this case, we are no longer in the identity case, so in order to apply the lifting formula, we need to apply the action of $(x\, y)$ to switch the lengths $\lambda_1$ and $\lambda_2$, inducing also an action on the triangulation $\cT$. The image of the curve and of $\cT$ under $(x\, y)$ is depicted in Figure \ref{fig:exaC3}. Now, we have to substitute $i=1,\, j=2, \, k=2$ in the lifting conditions for shape (C) obtaining

	\begin{figure}[h]
		\centering
		\begin{subfigure}[b]{.25\textwidth}
			\centering
			\includegraphics[width=0.725\textwidth]{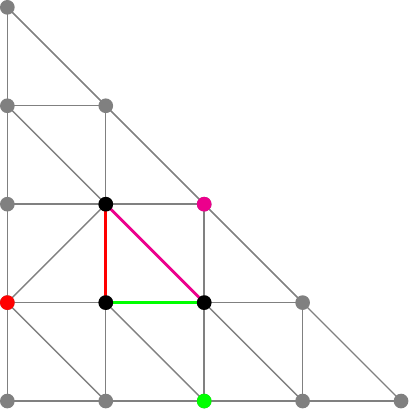}
			\caption{Image of Figure~\ref{fig:exaC1subdiv} under $(x\,y)$.}
		\end{subfigure}
		\begin{subfigure}[b]{.45\textwidth}
			\centering
			\includegraphics[width=0.55\textwidth]{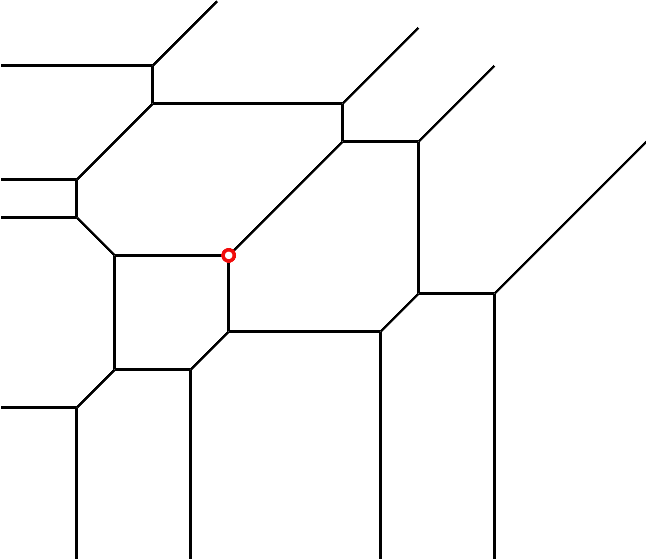}
			\caption{Image of Figure \ref{fig:exaC2} under $(x\,y)$.}
		\end{subfigure}
		\caption{The $(x\,y)$-transformation of the tropical curve in Figure \ref{fig:exaC2} and its dual triangulation. }\label{fig:exaC3}
	\end{figure}

\begin{align*}
	-s_{11}s_{12}s_{01}s_{20}>0 \text{ and } s_{22}s_{20}>0.
\end{align*} 
We then deduce  the lifting conditions for the original quartic with $\lambda_2<\lambda_1\leq\lambda_3$ by applying $(x\,y)^{-1}=(x\,y)$: 
\begin{equation}\label{eq:nongencase} -s_{11}s_{21}s_{10}s_{02}>0 \text{ and } s_{22}s_{02}>0.
\end{equation} 	
The second inequalities in (\ref{eq:gencase}) and (\ref{eq:nongencase}) are different. However, we observe that the first inequality $-s_{11}s_{21}s_{10}s_{02}>0$ is true if and only if $s_{02}=-s_{11}s_{21}s_{10}$. Substituting this equation into the second inequality, we see that the real lifting conditions are equivalent.
\end{exa}

\begin{prop}\label{prop:C}
	Let $\Gamma$ be a smooth tropical plane quartic curve 
	with dual triangulation $\cT$ and a bitangent class $B$ of shape (C). For every $c\in\Sigma(\cT)$,  the real lifting conditions of $B_c$ in $\Gamma_c$ are equivalent.
\end{prop}
\begin{proof}
We fix the following notation, see also Figure \ref{fig:subdivC}: $i$ is the $y$-coordinate of the  vertex $p_{0i}$, which forms a triangle with the (red) edge $\overline{p_{11}p_{12}}$, $j$ is the $x$-coordinate of the vertex $p_{j0}$, which forms a triangle with the (green) edge $\overline{p_{11}p_{21}}$ and $k$ is the $x$-coordinate of the vertex $p_{k,4-k}$, which forms a triangle with the (pink) edge $\overline{p_{12}p_{21}}$. 
\begin{figure}[h]
	\centering
	\includegraphics[width=0.15\textwidth]{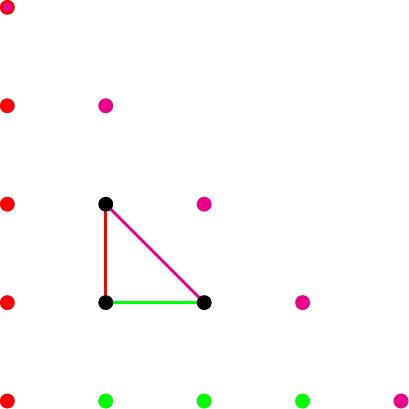}
	\caption{The dual deformation motif to shape (C) in identity position.}\label{fig:subdivC}
\end{figure}

We compute the real lifting conditions of shape (C)  for a bitangent class not in identity position. 
Suppose that $\Gamma_c$ has a unique shortest edge among $\lambda_1$, $\lambda_2$, $\lambda_3$. Then there exists $\sigma\in S_3$ such that for  $\sigma(\Gamma)$ the lattice lengths of the edges adjacent to $\sigma(B)$ satisfy $\lambda_1<\lambda_2\leq\lambda_3$. This corresponds to $\sigma(B)$ being in identity position. We can then determine the real lifting condition for $\sigma(B)$ using Table \ref{tab:tab11} and the parameters from $\sigma(\cT)$. In order to do this, we first need to look at the images of the three lattice points $p_{0i}$, $p_{j0}$ and $p_{k,4-k}$ under $\sigma$. Their images will lie in the boundary of $4\Delta_2$: $\sigma(p_{0i}),\,\sigma(p_{j0}),\,\sigma(p_{k,4-k})\in\{p_{0\tilde{i}},\,p_{\tilde{j}0},\,p_{\tilde{k},4-\tilde{k}} \}.$  Secondly, we substitute the values of the tilde indices into the lifting conditions, and then apply $\sigma^{-1}$ to obtain the real lifting conditions of $B=\sigma^{-1}(\sigma(B))$ in $\Gamma=\sigma^{-1}(\sigma(\Gamma))$.
Finally, we have to compare the lifting conditions of the bitangent class $B$ of shape (C) in $\Gamma$ with the ones of $B_c$ of shape (C) in $\Gamma_{c}$ where $c\in\Sigma(\cT)$ such that the dual deformation motif of $(\Gamma_{c},B_{c})$ is in identity position and $\Gamma_{c}$ satisfies $\lambda_1<\lambda_2\leq\lambda_3$. If the conditions are equivalent, we have proven that the real lifting conditions do not change.
Since $S_3$ is generated by $(x\,y)$ and $(x\,z)$, it suffices to check the  cases  $\lambda_2<\lambda_1\leq\lambda_3$ and $\lambda_3<\lambda_2\leq\lambda_1.$
\\
If the tropical quartic satisfies $\lambda_2<\lambda_1\leq\lambda_3$, we apply $\sigma=(x\,y)$ to obtain a generic representative of the identity position. 
Now, $\sigma(p_{0i})=p_{i0}$ and $\sigma(p_{j0})=p_{0j}$ and $\sigma(p_{k,4-k})=p_{4-k,k}$, so we have to substitute 
\begin{align*}
	i \mapsto j,\,\,
	j \mapsto i,\,\,
	k \mapsto 4-k
\end{align*}
in the real lifting conditions in Table \ref{tab:tab11}. Here the values of $i$, $j$, $k$ are as in the quartic we started with. Thus, the real lifting conditions for $\sigma(B)$ are 
\begin{center}
	\begin{tabular}{ccc}
		$i=1,3 $& $(-s_{11})^{j+1}s_{12}^{j} s_{21}s_{0j}s_{i0}>0$ 		& $(-s_{21})^{4-k+1}s_{12}^{4-k}s_{11}s_{4-k,k}s_{i0}>0$,\\
		$ i=2$ &$ (-s_{11}s_{12})^{j}s_{0j}s_{20}>0$		& $(-s_{12}s_{21})^{4-k}s_{4-k,k}s_{20}>0$.
	\end{tabular}
\end{center}
We now apply $(x\,y)^{-1}=(x\,y)$. Note that this acts on  the indices of the signs but not on the exponents. We obtain the conditions
\begin{center}
	\begin{tabular}{ccc}
		$i=1,3 $& $(-s_{11})^{j+1}s_{21}^{j} s_{12}s_{j0}s_{0i}>0$ & $(-s_{12})^{4-k+1}s_{21}^{4-k}s_{11}s_{k,4-k}s_{0i}>0$,\\
		$ i=2$ &$ (-s_{11}s_{21})^{j}s_{j0}s_{02}>0$ 		& $(-s_{21}s_{12})^{4-k}s_{k,4-k}s_{02}>0.$	\\
	\end{tabular}	
\end{center}

Now, we suppose that we have a quartic for which we can, by edge length changes, switch between the cases $\lambda_1<\lambda_2\leq\lambda_3$ and $\lambda_2<\lambda_1\leq\lambda_3$. For such a quartic the dual triangulation must satisfy $i,j\in\{1,2,3\}$, otherwise the genericity condition is not satisfied. We compare the real lifting conditions for the different cases after simplifying some exponents:
\begin{center}
	\begin{tabular}{c|c|c}
		& $\lambda_1<\lambda_2\leq\lambda_3$ & $\lambda_2<\lambda_1\leq\lambda_3$\\
		\hline
		$i,j \in{1,3}$ & $s_{12} s_{21}s_{0i}s_{j0}>0$ & $s_{21} s_{12}s_{j0}s_{0i}>0$\\
		& $(-s_{21})^{k+1}s_{12}^ks_{11}s_{k,4-k}s_{j0}>0$& $(-s_{12})^{k+1}s_{21}^{k}s_{11}s_{k,4-k}s_{0i}>0$ \\
		\hline
		$i=j=2$ &  $ s_{02}s_{20}>0$ & $ s_{20}s_{02}>0$\\
		& $(-s_{12}s_{21})^ks_{k,4-k}s_{20}>0$&$(-s_{21}s_{12})^{k}s_{k,4-k}s_{02}>0$ \\
		\hline
		$j=2, i\in\{1,3\}$ & $-s_{11}s_{12}s_{0i}s_{20}>0$& $-s_{11} s_{12}s_{20}s_{0i}>0$\\
		& $(-s_{12}s_{21})^ks_{k,4-k}s_{20}>0$ & $(-s_{12})^{k+1}s_{21}^{k}s_{11}s_{k,4-k}s_{0i}>0$\\
		\hline
		$j\in\{1,3\}, i=2$& $-s_{11} s_{21}s_{02}s_{j0}>0$ & $ -s_{11}s_{21}s_{j0}s_{02}>0$\\
		& $(-s_{21})^{k+1}s_{12}^ks_{11}s_{k,4-k}s_{j0}>0$&  $(-s_{21}s_{12})^{k}s_{k,4-k}s_{02}>0$\\
	\end{tabular}
\end{center}
We see that in each case the first inequalities are the same. The second inequalities differ, but it can be shown that they are equivalent by taking the first inequalities into account. 
It follows that for the edge length change between $\lambda_1<\lambda_2\leq\lambda_3$ and $\lambda_2<\lambda_1\leq\lambda_3$ the real lifting conditions for shape (C) do not change.

The last step to complete the proof is to consider $\Gamma$ such that $\lambda_1<\lambda_3\leq\lambda_2$. To obtain the real lifting conditions for (C), we have to apply $\sigma=(x\,z)$ to the subdivision and the curve to obtain a generic representative of the identity position.
We apply $\sigma$ to the lattice points $p_{0i}$, $p_{j0}$ and $p_{k,4-k}$ to obtain the values that we have to substitute in the real lifting conditions for $\sigma(B)$ and we obtain
\begin{align*}
	i \mapsto 4-k,\,\,
	j \mapsto 4-j,\,\,
	k \mapsto 4-i.
\end{align*}
So the lifting conditions for $\sigma(B)$ are given as \\
\begin{center}
	\begin{tabular}{ccc}
		$4-j=1,3$ & $(-s_{11})^{4-k+1}s_{12}^{4-k} s_{21}s_{0,4-k}s_{4-j,0}>0$ &  $(-s_{21})^{4-i+1}s_{12}^{4-i}s_{11}s_{4-i,i}s_{4-j,0}>0$,\\
		$4-j=2$ & $(-s_{11}s_{12})^{4-k}s_{0,4-k}s_{20}>0$ 	&	 $(-s_{12}s_{21})^{4-i}s_{4-i,i}s_{20}>0.$
	\end{tabular}
\end{center}
Applying $(x\,z)^{-1}=(x\,z)$ to these gives the real lifting conditions of $B$ when $\lambda_1<\lambda_3\leq\lambda_2$:

\begin{center}
	\begin{tabular}{ccc}
		$j=1,3$ & $(-s_{21})^{k+1}s_{12}^{k} s_{11}s_{k,4-k}s_{j0}>0$ &  $(-s_{11})^{i+1}s_{12}^{i}s_{21}s_{0i}s_{j0}>0$,\\
		$j=2$ & $(-s_{21}s_{12})^{k}s_{k,4-k}s_{20}>0$ 	&	 $(-s_{12}s_{11})^{i}s_{0i}s_{20}>0$.
	\end{tabular}
\end{center}

%Now we have to prove that changing the edge lengths between the two cases $\lambda_1<\lambda_2\leq \lambda_3$ and  $\lambda_1<\lambda_3\leq \lambda_2$ does not change the lifting conditions:
%\begin{center}
%	\begin{tabular}{c|c|c}
%		& $\lambda_1<\lambda_2\leq\lambda_3$ & $\lambda_1<\lambda_3\leq\lambda_2$\\
%		\hline
%		$j \in{1,3}$ &  $(-s_{11})^{i+1}s_{12}^i s_{21}s_{0i}s_{j0}>0$ &  $(-s_{21})^{k+1}s_{12}^{k} s_{11}s_{k,4-k}s_{j0}>0$\\
%		& $(-s_{21})^{k+1}s_{12}^ks_{11}s_{k,4-k}s_{j0}>0$&  $(-s_{11})^{i+1}s_{12}^{i}s_{21}s_{0i}s_{j0}>0$ \\
%		\hline
%		$j=2$ & $(-s_{11}s_{12})^is_{0i}s_{20}>0$ &$(-s_{21}s_{12})^{k}s_{k,4-k}s_{20}>0$\\
%		& $(-s_{12}s_{21})^ks_{k,4-k}s_{20}>0$& $(-s_{12}s_{11})^{i}s_{0i}s_{20}>0$ 	
%	\end{tabular}
%\end{center}
Similar comparisons as before show us that the lifting conditions are the same as for  $\lambda_1<\lambda_2\leq \lambda_3$.
\end{proof}

\begin{theorem}\label{theorem:lifting}
Let $\Gamma$ be  a generic tropical smooth quartic curve with dual triangulation $\cT$, and let $B$ be a tropical bitangent class.  For every $c\in\Sigma(\cT)$, the real lifting conditions of $B_c$ in $\Gamma_c$ are independent on the shape of the bitangent class.  In other words: real lifting conditions of tropical bitangent classes only depend on the dual subdivision $\cT$ of $\Gamma$.
\end{theorem}
\begin{proof}
	We prove this by going through all deformation classes not considered in Lemma~ \ref{prop:signsEFJ} and \ref{prop:signsGKUTV} and Proposition \ref{prop:C}  and by comparing the lifting conditions of the different shapes. The shapes in  the deformation classes (W X Y EE GG) and (W ... HH)  have no real lifting conditions, so the statement holds. 	The remaining deformation classes are summarized below. 

	\begin{longtable}{m{5cm}|m{8.5cm}}		
		\hline\multicolumn{2}{c}{\textbf{(B H' H), (B H' H)$+(y\,z)$, (B M)$+(y\,z)$ } }\\
		\hline 
		\multicolumn{2}{l}{
			(B): \hspace{0.8cm}$(-s_{1v}s_{1,v+1})^{i+1} s_{0i}s_{21}>0$ and $ (-s_{21})^{j+1}s_{31}^js_{1v}s_{1,v+1}s_{j0}>0$}\\
		\multicolumn{2}{l}{	(H),(H'):  $(-s_{1v}s_{1,v+1})^{i+1} s_{0i}s_{21}>0 $ and $-s_{21}s_{1v}s_{1,v+1}s_{40}>0$	
		}\\
		\multicolumn{2}{l}{	(M): \hspace{0.6cm} $(-s_{1v}s_{1,v+1})^{i+1} s_{0i}s_{21}>0$ and $s_{31}s_{1v}s_{1,v+1}s_{30}>0$	
		}\\
		\hline\multicolumn{2}{c}{\textbf{(B H' H)} }\\
		\hline
		\includegraphics[width = 0.15\textwidth]{figures/defclass/_B_-_H__-_H_.pdf} & 
		The dual bitangent motif of (B) yields $j=4$. Substituting it gives the same real lifting conditions for (B), (H') and (H).\\
		\hline
		\newpage
		\hline
		
		\multicolumn{2}{c}{\textbf{(B H' H)$+(y\,z)$}}\\ \hline
		\includegraphics[width = 0.15\textwidth]{figures/defclass/_B_-_H__-_H_-_H__-_B_.pdf} & The lifting conditions for (H) and (H)$_{(y\,z)}$ are the same, so the equality of the real lifting conditions follows from the statement for (B H' H).\\
		\hline
		\multicolumn{2}{c}{\textbf{(B M)$+(y\,z)$ }}\\\hline
		\includegraphics[width = 0.15\textwidth]{figures/defclass/_B_-_M_-_B_.pdf} & The dual bitangent motif of (B)  yields $j=3$. Substituting it gives the same real lifting conditions for (B) and (M). The real lifting conditions for (M) and (M)$_{(y\,z)}$ are the same, so  the  statement follows from the first part. \\
		\hline
		\multicolumn{2}{c}{ }\\ \hline
		\multicolumn{2}{c}{\textbf{(D L' Q), (D L' Q Q' R), (D L O) }}\\ 
		\hline 
		\multicolumn{2}{l}{
			(D): \hspace{1.8cm} $(-s_{10}s_{11})^is_{0i}s_{22}>0$}
		\\
		\multicolumn{2}{l}{	(L'),(Q),(Q');(R):  $s_{00}s_{22}>0$	
		}\\
		\multicolumn{2}{l}{	(L),(O):\hspace{1.5cm} $-s_{10}s_{11}s_{01}s_{22}>0$
		}\\
		\hline
		\multicolumn{2}{c}{\textbf{(D L' Q) }}\\ \hline
		\includegraphics[width = 0.15\textwidth]{figures/defclass/_D_-_L__-_Q_.pdf} & The dual bitangent motif of shape (D) yields $i=0$, so the real lifting conditions for all shapes in this deformation class coincide.\\ \hline
		\multicolumn{2}{c}{\textbf{(D L' Q Q' R)}} \\ \hline\includegraphics[width = 0.15\textwidth]{figures/defclass/_D_-_L__-_Q_-_Q__-_R_} & Same argument as for (D L' Q).\\ \hline
		\multicolumn{2}{c}{\textbf{(D L O)}}\\ \hline
		\includegraphics[width = 0.15\textwidth]{figures/defclass/_D_-_L_-_O_.pdf} & The dual bitangent motif of shape (D)  yields $i=1$, so the real lifting conditions for all shapes in this deformation class coincide.\\ \hline

		\multicolumn{2}{c}{ }\\ \hline
		\multicolumn{2}{c}{\textbf{(G I N)$+(x\,y)$, (G K U T T')}} \\ \hline
		\multicolumn{2}{l}{	(G):\hspace{4.2cm} $(-s_{10}s_{11})^is_{0i}s_{k,4-k}>0$}
		\\
		\multicolumn{2}{l}{	(I),(N): \hspace{3.7cm} $-s_{10}s_{11}s_{01}s_{k,4-k}>0$	
		}\\
		\multicolumn{2}{l}{	(K),(T),(T'),(T''), (U),(U'),(V): $s_{00}s_{k,4-k}>0$ 		
		}\\
		\hline		

		\multicolumn{2}{c}{\textbf{(G I N)$+(x\,y)$}}\\ \hline
		
		\includegraphics[width = 0.15\textwidth]{figures/defclass/_G_-_I_-_N_-_I_-_G_1.pdf}  \includegraphics[width = 0.15\textwidth]{figures/defclass/_G_-_I_-_N_-_I_-_G_2.pdf} & The dual bitangent motif of shape (G) yields $i=1$, so the real lifting conditions of (G),  (I) and (N) coincide. Since (N) and (N)$_{(x\,y)}$ have the same real lifting conditions, the  statement follows from the first part.\\ \hline
		\multicolumn{2}{c}{\textbf{(G K U T T')}}\\ \hline
		\includegraphics[width = 0.15\textwidth]{figures/defclass/_G_-_K_-_U_-_T_-_T__1.pdf} \includegraphics[width = 0.15\textwidth]{figures/defclass/_G_-_K_-_U_-_T_-_T__2.pdf} &The dual bitangent motif of shape (G)  yields $i=0$, so the real lifting conditions for all shapes in this deformation class coincide.\\ \hline
	\end{longtable}
\end{proof}	
\begin{rem}
Markwig, Payne and Shaw characterize the lifting conditions of bitangent shapes over arbitrary fields \cite{MPS21}.
	An analogous investigation to the proof of Theorem~\ref{theorem:lifting} shows that the lifting conditions again stay constant in every deformation class. Therefore, deformation classes are relevant for the lifting behavior of tropical bitangents over arbitrary fields, not only over real closed fields. 
\end{rem}

Table \ref{tab:tab11new} summarizes the real lifting conditions for the deformation classes.
\begin{small}
\begin{longtable}{c|c}
deformation class & Lifting conditions \\
\hline\endhead
(A) & $(-s_{1v}s_{1,v+1})^i s_{0i}s_{22}>0$ and $(-s_{u1}s_{u+1,1})^js_{j0}s_{22}>0$ \\
\hline
(B H' H), & \\
(B H' H)+$(yz)$, &  $(-s_{1v}s_{1,v+1})^{i+1} s_{0i}s_{21}>0$ and $-s_{21}s_{1v}s_{1,v+1}s_{40}>0$ \\
(H) & \\
\hline
(B M)+$(yz)$ & $(-s_{1v}s_{1,v+1})^{i+1} s_{0i}s_{21}>0$ and $s_{31}s_{1v}s_{1,v+1}s_{30}>0$\\
\hline
(B) & $(-s_{1v}s_{1,v+1})^{i+1} s_{0i}s_{21}>0$ and $(-s_{21})^{j+1}s_{31}^js_{1v}s_{1,v+1}s_{j0}>0$ \\& with $j\in\{0,1,2\}$\\
\hline
(C) & $(-s_{11}s_{12})^is_{0i}s_{20}>0$ and $(-s_{21}s_{12})^k s_{k,4-k}s_{20}>0\,\,\,$ if j=2\\
& $(-s_{11})^{i+1}s_{12}^is_{21}s_{0i}s_{j0}>0$ and $(-s_{21})^{k+1}s_{12}^ks_{11} s_{k,4-k}s_{j0}>0$ if j=1,3\\
\hline		
(D) & $(-s_{10}s_{11})^is_{0i}s_{22}>0$ with $i\in\{2,3,4\}$\\
\hline
(D L O), (P) & $-s_{10}s_{11}s_{01}s_{22}>0$ \\
\hline
(D L' Q),& \\
(D L' Q Q' R), & $s_{00}s_{22}>0$\\
(S), (T) & \\
\hline
(E), (E F J) & $(-s_{1v}s_{1,v+1})^{i+1} s_{0i}s_{20}>0$\\
\hline
(G) & $(-s_{10}s_{11})^is_{0i}s_{k,4-k}>0$ with $i\in\{2,3,4\}$\\
\hline
(G I N)+$(xy)$ &$-s_{10}s_{11}s_{01}s_{k,4-k}>0$ \\
\hline
(G K U U' T T'), &  \\
(G K U U' T- & $s_{00}s_{k,4-k}>0$\\
-  T' T'' V)+$(xy)$ & \\
\hline
rest & no conditions\\
\caption{Real lifting conditions of the deformation classes in their positions as in Figure \ref{fig:defclasses}.}\label{tab:tab11new}
\end{longtable}
\end{small}

We are now ready to give a proof of our main result. 

\smallskip 
\noindent
\emph{Proof of Theorem \ref{thm:Pluecker}.} Let $\Gamma$ be a generic tropicalization of a tropically smooth quartic curve $V(f)$ defined over $\KK_{\RR}$ and $\cT$ its dual triangulation.
	By Theorem \ref{theorem:lifting}, the real lifting conditions of the $7$ bitangent classes of $\Gamma$ only depend on their $7$ deformation classes. Furthermore, the deformation classes are uniquely determined by their dual deformation motifs in the triangulation $\cT$, as pointed out in Remark \ref{rem:dualmotifs}. Therefore, the real lifting conditions for the tropical bitangent classes of $\Gamma$ only depend on the triangulation $\cT$.
	
In order to prove the statement, we need to enumerate the dual deformation motifs of the deformation classes in the $1278$ unimodular regular triangulations of  $4\Delta_2$ modulo $S_3$-symmetry as computed in \cite{BJMS15}. Of these $S_3$-representatives  exactly eight do not satisfy the genericity constraint that a vertex of the curve with adjacent edges of directions $-e_1$, $-e_2$, $e_1+e_2$ needs to have a unique shortest adjacent edge.
For these eight cones we did run the same computations as for the generic cones, but we could not compute the lifting behavior of the bitangent class of shape (C) since this is not yet understood. However, since the numbers of real bitangents are already known classically, our computations might help to understand these special cases.

We implemented the search for the dual deformation motifs of the deformation classes in \polymake \cite{polymake:2000}. For each deformation class, we considered the real lifting conditions determined in Theorem \ref{theorem:lifting} and summarized in Table \ref{tab:tab11new}, and we evaluated them for all possible $2^{15}$ sign vectors. Again, we implemented this in \polymake obtaining that each sign vector satisfies the lifting conditions of  $1$, $2$, $4$ or $7$ deformation classes. A more detailed description of the computational procedure and codes can be found in \cite{2GP21}. 

	Finally, by \cite[Theorem 1.2 and Corollary 7.3]{CueMa20}, each bitangent class has either zero or exactly four lifts to totally real bitangents. Therefore, the smooth quartic curve $V(f)$ with smooth tropicalization $\Gamma$ has either $4$, $8$, $16$ or $28$ totally real bitangents. \hfill$\Box$

\clearpage

\appendix\section{}\label{sec:Appendix}
\noindent
In this section we provide further figures to help 
understanding the classification statement in Section \ref{sec:classification}. We recommend considering the figures together with the table in the proof of Theorem \ref{thm:classification}.

\subsection{Example of deformation class (B H' H)}
Let $\Gamma$ be a smooth tropical quartic curve  with a bitangent class $B$ 
 with dual bitangent motif in identity position contained in the subcomplex in Figure \ref{fig:subdivB-Hchanges1}. Note that  we are choosing the red edge $\overline{p_{12}p_{13}}$ in the corresponding picture in Figure \ref{fig:defclasses}. A similar example can be drawn for $\overline{p_{11}p_{12}}$.
\begin{figure}[h]
	\centering
	\begin{subfigure}[b]{0.2\textwidth}
		\centering
		\includegraphics[width=0.7\textwidth]{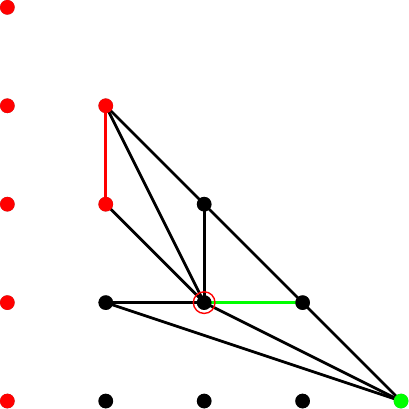}
		\caption*{}
	\end{subfigure}
		\begin{subfigure}[b]{.2\textwidth}
		\centering\includegraphics[width=0.75\textwidth]{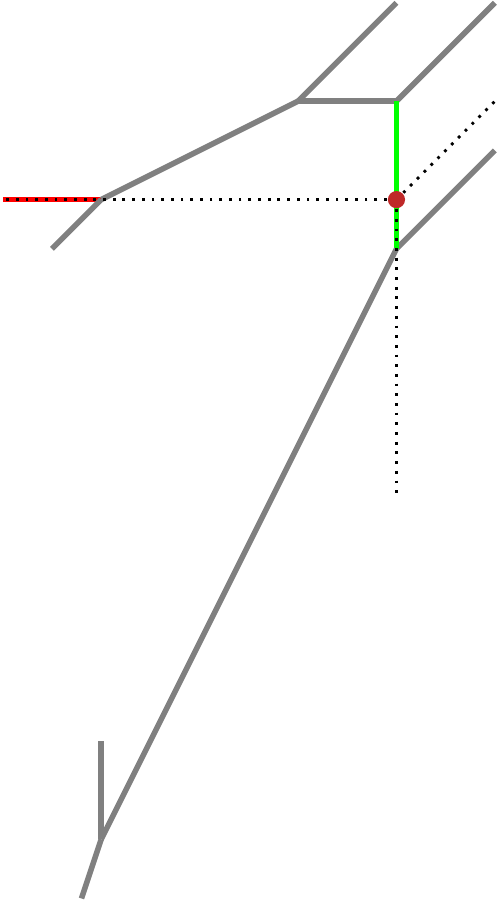}
		\caption*{(B)}
	\end{subfigure}
	\begin{subfigure}[b]{.2\textwidth}
		\centering\includegraphics[width=0.77\textwidth]{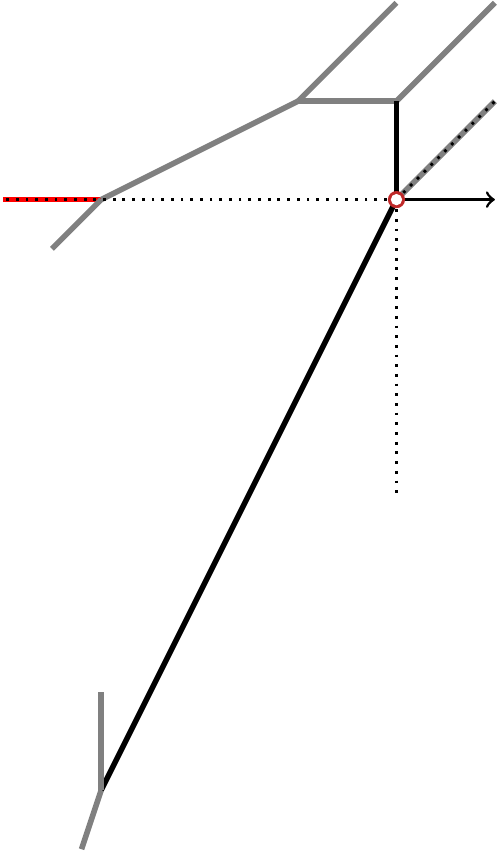}
		\caption*{(H')}
	\end{subfigure}
	\begin{subfigure}[b]{.2\textwidth}
		\centering\includegraphics[width=0.8\textwidth]{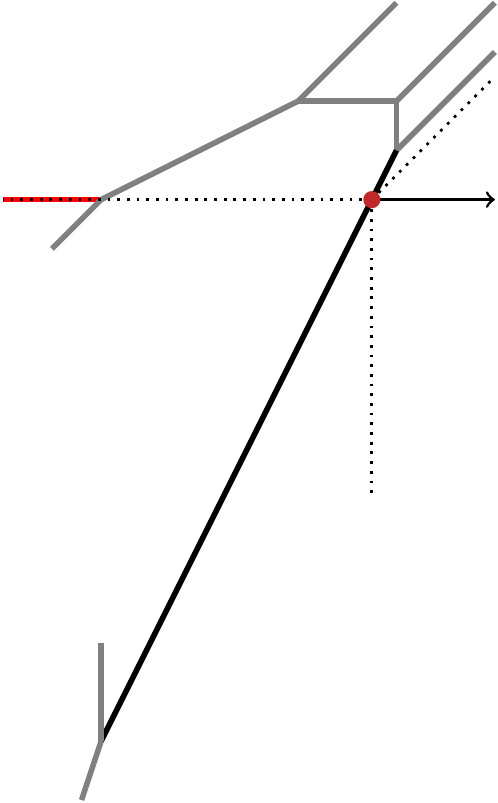}
		\caption*{(H)}
	\end{subfigure}	
	\caption{Example of deformation class (B H' H).	}\label{fig:subdivB-Hchanges1}
\end{figure}

\subsection{Example of deformation class (B H' H)$+(y\,z)$}
Let $\Gamma$ be a smooth tropical quartic curve  with a bitangent class $B$ 
with dual bitangent motif in identity position contained in the subcomplex in Figure \ref{fig:subdivBHchanges2}.  Note that  we are choosing the red edge $\overline{p_{12}p_{13}}$ in the corresponding picture in Figure \ref{fig:defclasses}. Similar examples can be drawn for $\overline{p_{10}p_{11}}$ and $\overline{p_{11}p_{12}}$.

\begin{figure}[h]
	\centering
\begin{subfigure}[b]{.2\textwidth}
	\centering
	\includegraphics[width=0.7\textwidth]{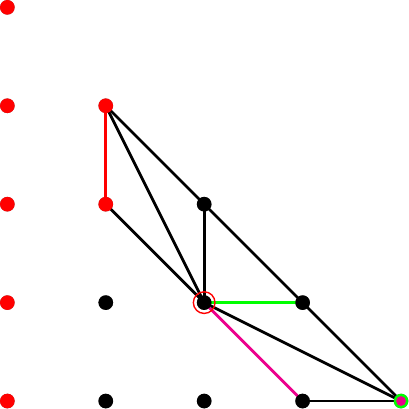}
	\caption*{}
\end{subfigure}
	\begin{subfigure}[b]{.3\textwidth}
		\centering\includegraphics[width=0.75\textwidth]{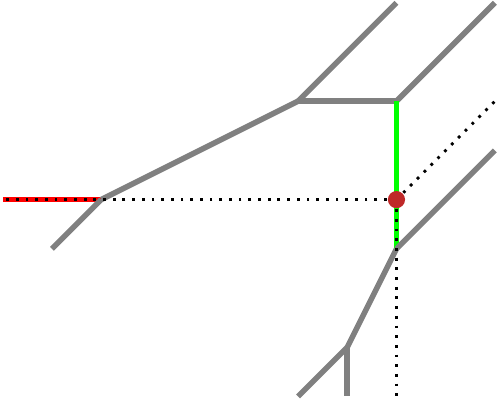}
		\caption*{(B)}
	\end{subfigure}
	\begin{subfigure}[b]{.3\textwidth}
		\centering\includegraphics[width=0.75\textwidth]{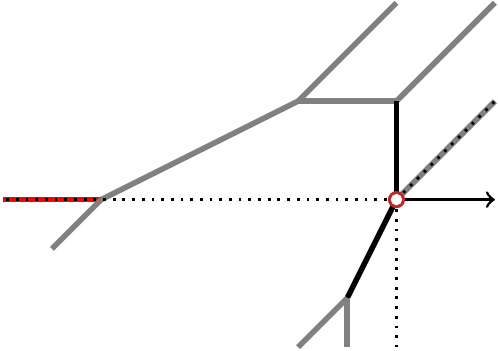}
		\caption*{(H')}
	\end{subfigure}
	\begin{subfigure}[b]{.3\textwidth}
		\centering\includegraphics[width=0.75\textwidth]{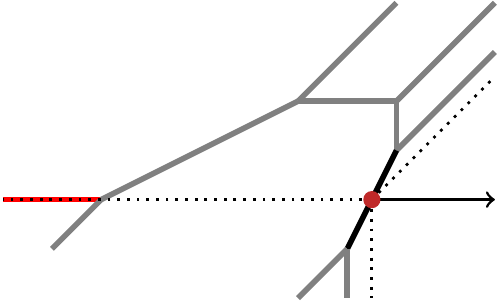}
		\caption*{(H)}
	\end{subfigure}
	\begin{subfigure}[b]{.3\textwidth}
		\centering\includegraphics[width=0.75\textwidth]{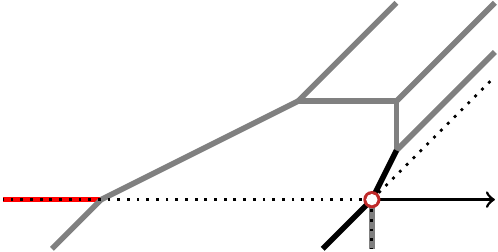}
		\caption*{(H')$_{(y\,z)}$}
	\end{subfigure}
	\begin{subfigure}[b]{.3\textwidth}
		\centering\includegraphics[width=0.75\textwidth]{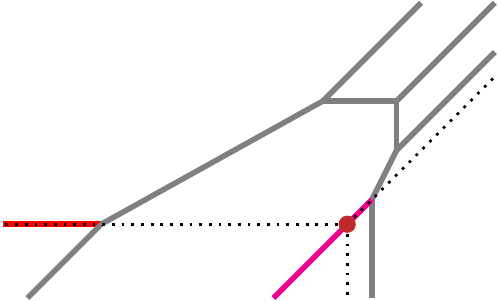}
		\caption*{(B)$_{(y\,z)}$}
	\end{subfigure}
		\caption{Example of deformation class (B H' H)$+(y\,z)$.	}\label{fig:subdivBHchanges2}
\end{figure}

\subsection{Example of deformation class (B M)$+(y\,z)$}
Let $\Gamma$ be a smooth tropical quartic curve  with a bitangent class $B$ 
 with dual bitangent motif in identity position contained in the subcomplex in Figure \ref{fig:subdivBMchanges1}.  Note that  we are choosing the red edge $\overline{p_{12}p_{13}}$ in the corresponding picture in Figure \ref{fig:defclasses}. Similar examples can be drawn for $\overline{p_{10}p_{11}}$ and $\overline{p_{11}p_{12}}$.
\begin{figure}[h]
	\centering
	\begin{subfigure}[b]{.14\textwidth}
	\centering
	\includegraphics[width=0.9\textwidth]{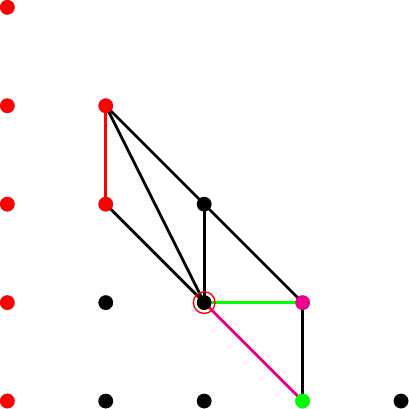}
	\caption*{}
\end{subfigure}
	\begin{subfigure}[b]{.26\textwidth}
		\centering\includegraphics[width=0.9\textwidth]{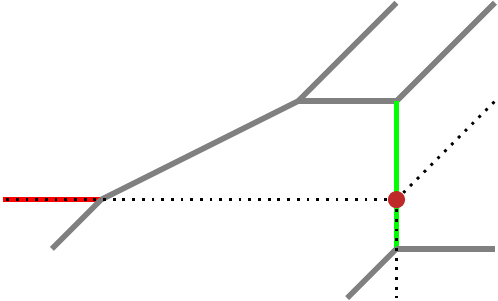}
		\caption*{(B)}
	\end{subfigure}
	\begin{subfigure}[b]{.26\textwidth}
		\centering\includegraphics[width=0.9\textwidth]{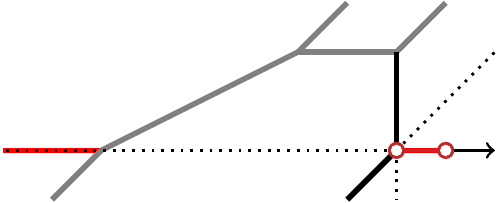}
		\caption*{(M)}
	\end{subfigure}
	\begin{subfigure}[b]{.26\textwidth}
		\centering\includegraphics[width=0.9\textwidth]{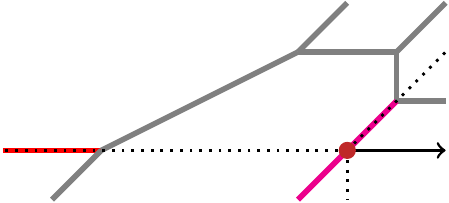}
		\caption*{(B)$_{(y\,z)}$}
	\end{subfigure}
	\caption{Example of deformation class (B M)$+(y\,z)$. }\label{fig:subdivBMchanges1}
\end{figure}

\subsection{Example of deformation class (D L' Q)}
Let $\Gamma$ be a smooth tropical quartic curve  with a bitangent class $B$ 
with dual bitangent motif in identity position contained in the subcomplex in Figure \ref{fig:subdivDL'Qchanges1}.  
\begin{figure}[h]
	\centering
		\begin{subfigure}{.14\textwidth}
		\centering\includegraphics[width=0.9\textwidth]{figures/defclass/_D_-_L__-_Q_.pdf}
		\caption*{}
	\end{subfigure}
	\begin{subfigure}{.25\textwidth}
		\centering\includegraphics[width=0.9\textwidth]{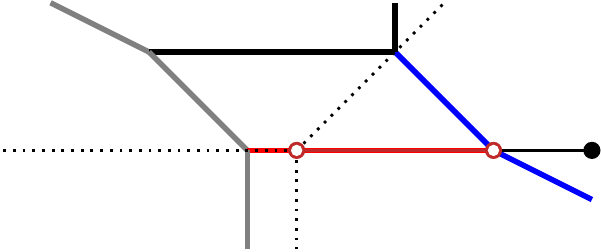}
		\caption*{(D)}
	\end{subfigure}
	\begin{subfigure}{.25\textwidth}
		\centering\includegraphics[width=0.9\textwidth]{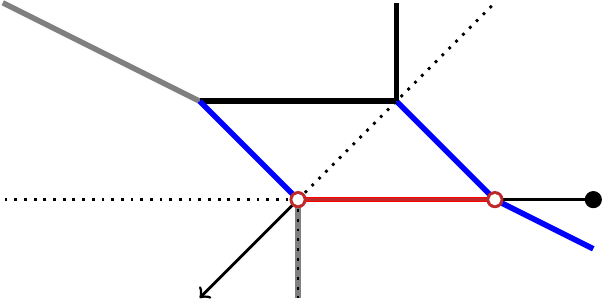}
		\caption*{(L')}
	\end{subfigure}
	\begin{subfigure}{.25\textwidth}
		\centering\includegraphics[width=0.9\textwidth]{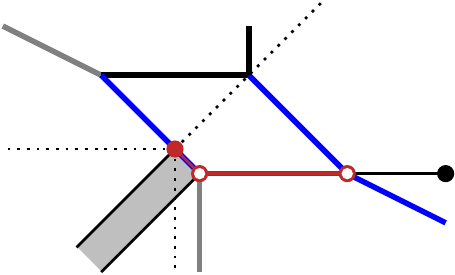}
		\caption*{(Q)}
	\end{subfigure}
	\caption{Example of deformation class (D L' Q).}\label{fig:subdivDL'Qchanges1}
\end{figure}

\subsection{Example of deformation class  (D L' Q Q' R)}
Let $\Gamma$ be a smooth tropical quartic curve  with a bitangent class $B$ 
 with dual bitangent motif in identity position contained in the subcomplex in Figure \ref{fig:subdivDL'QQ'Rchanges1}. 
\begin{figure}[h]
	\centering
			\begin{subfigure}[b]{.14\textwidth}
		\centering\includegraphics[width=0.9\textwidth]{figures/defclass/_D_-_L__-_Q_-_Q__-_R_.pdf}
		\caption*{}
	\end{subfigure}
	\begin{subfigure}[b]{.16\textwidth}
		\centering\includegraphics[width=1\textwidth]{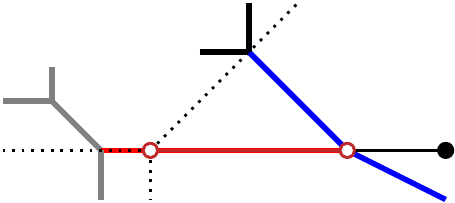}
		\caption*{(D)}
	\end{subfigure}
	\begin{subfigure}[b]{.16\textwidth}
		\centering\includegraphics[width=1\textwidth]{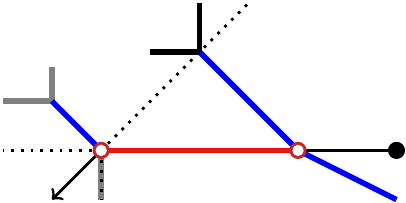}
		\caption*{(L')}
	\end{subfigure}
	\begin{subfigure}[b]{.16\textwidth}
		\centering\includegraphics[width=1\textwidth]{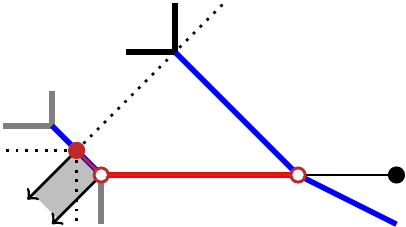}
		\caption*{(Q)}
	\end{subfigure}
\begin{subfigure}[b]{.16\textwidth}
	\centering\includegraphics[width=1\textwidth]{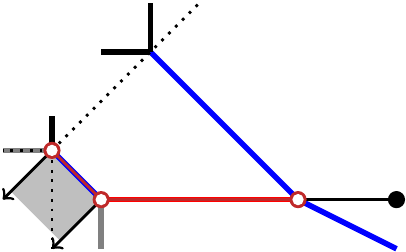}
	\caption*{(Q')}
\end{subfigure}
\begin{subfigure}[b]{.16\textwidth}
	\centering\includegraphics[width=1\textwidth]{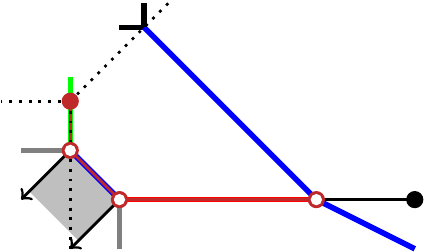}
	\caption*{(R)}
\end{subfigure}
	\caption{Example of deformation class  (D L' Q Q' R).}\label{fig:subdivDL'QQ'Rchanges1}
\end{figure}

\subsection{Example of deformation class (D L O)}
Let $\Gamma$ be a smooth tropical quartic curve  with a bitangent class $B$ 
 with dual bitangent motif in identity position contained in the subcomplex in Figure \ref{fig:subdivDLOchanges1}. 
\begin{figure}[h]
	\centering
	\begin{subfigure}[b]{.14\textwidth}
		\centering\includegraphics[width=0.9\textwidth]{figures/defclass/_D_-_L_-_O_.pdf}
		\caption*{}
	\end{subfigure}
	\begin{subfigure}[b]{.26\textwidth}
		\centering\includegraphics[width=1\textwidth]{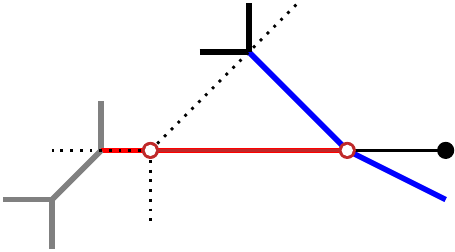}
		\caption*{(D)}
	\end{subfigure}
	\begin{subfigure}[b]{.26\textwidth}
		\centering\includegraphics[width=1\textwidth]{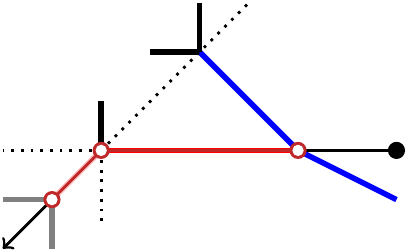}
		\caption*{(L)}
	\end{subfigure}
	\begin{subfigure}[b]{.26\textwidth}
		\centering\includegraphics[width=1\textwidth]{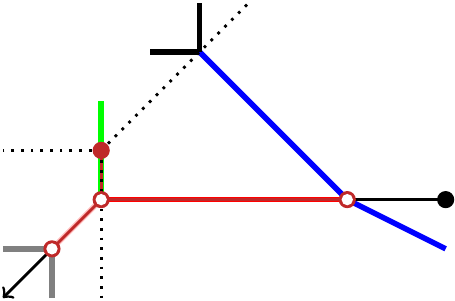}
		\caption*{(O)}
	\end{subfigure}
\caption{Example of deformation class  (D L O).}\label{fig:subdivDLOchanges1}
\end{figure}

\subsection{Deformation class (G I N)$+(x\,y)$}
For this deformation class there are two different cases of how the second tangency arises.
Let $\Gamma$ be a smooth tropical quartic curve  with a bitangent class $B$ 
with dual bitangent motif in identity position contained in the subcomplex in Figure \ref{fig:subdivGINIG1}. 
This figure depicts the case where one tangency is given by the blue edge $\overline{p_{11}p_{22}}$.
\begin{figure}[h]
	\centering
	\begin{subfigure}[b]{.15\textwidth}
		\centering\includegraphics[width=0.95\textwidth]{figures/defclass/_G_-_I_-_N_-_I_-_G_1.pdf}
		\caption*{}
	\end{subfigure}
	\begin{subfigure}[b]{.18\textwidth}
		\centering\includegraphics[width=1\textwidth]{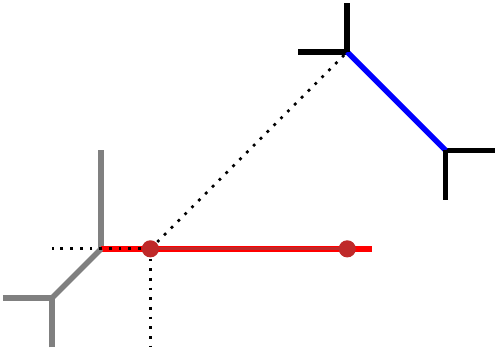}
		\caption*{(G)}
	\end{subfigure}
	\begin{subfigure}[b]{.18\textwidth}
		\centering\includegraphics[width=1\textwidth]{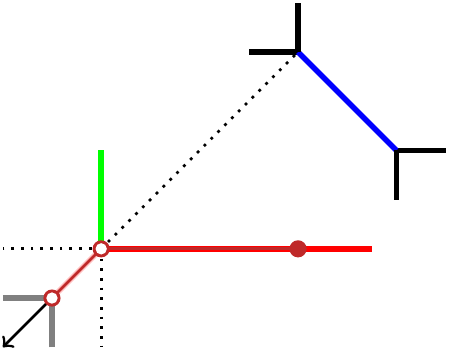}
		\caption*{(I)}
	\end{subfigure}
	\begin{subfigure}[b]{.16\textwidth}
		\centering\includegraphics[width=1\textwidth]{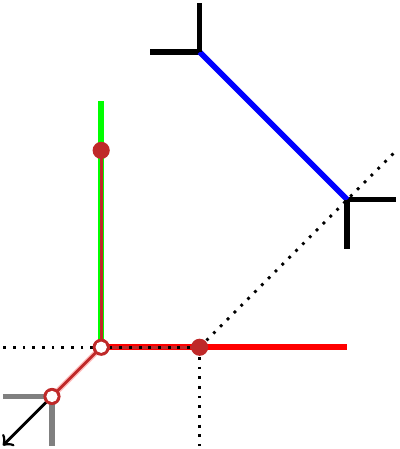}
		\caption*{(N)}
	\end{subfigure}
	\begin{subfigure}[b]{.14\textwidth}
		\centering\includegraphics[width=1\textwidth]{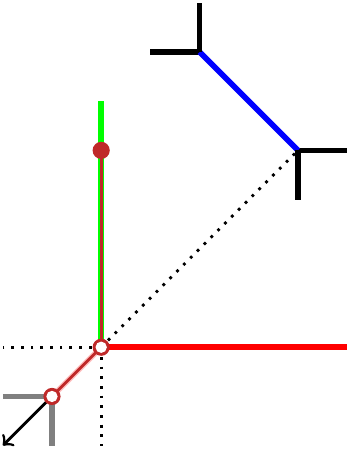}
		\caption*{(I)$_{(x\,y)}$}
	\end{subfigure}
	\begin{subfigure}[b]{.14\textwidth}
		\centering\includegraphics[width=1\textwidth]{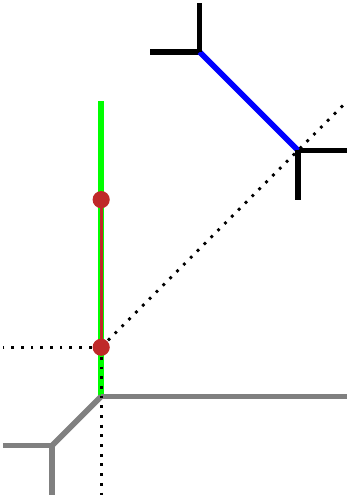}
		\caption*{(G)$_{(x\,y)}$}
	\end{subfigure}
	\caption{Example of deformation class  (G I N)$+(x\,y)$ with one tangency given by $\overline{p_{11}p_{22}}$.}\label{fig:subdivGINIG1}
\end{figure}

\noindent The situation is similar if we choose the other blue edge $\overline{p_{11}p_{04}}$, as illustrated in Figure~\ref{fig:subdivGINIG2}.
\begin{figure}[h]
	\centering
	\begin{subfigure}[b]{.15\textwidth}
		\centering\includegraphics[width=0.95\textwidth]{figures/defclass/_G_-_I_-_N_-_I_-_G_2.pdf}
		\caption*{}
	\end{subfigure}
	\begin{subfigure}[b]{.18\textwidth}
		\centering\includegraphics[width=1\textwidth]{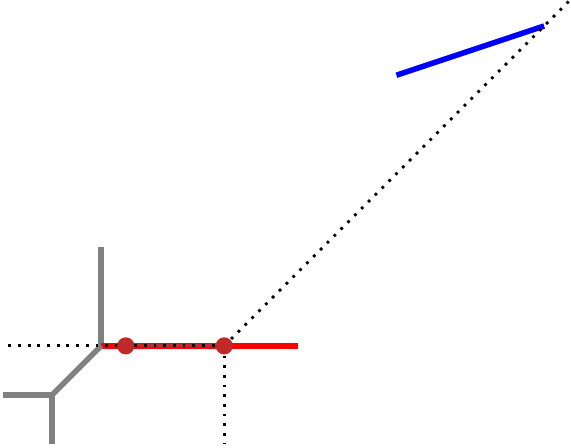}
		\caption*{(G)}
	\end{subfigure}
	\begin{subfigure}[b]{.17\textwidth}
		\centering\includegraphics[width=1\textwidth]{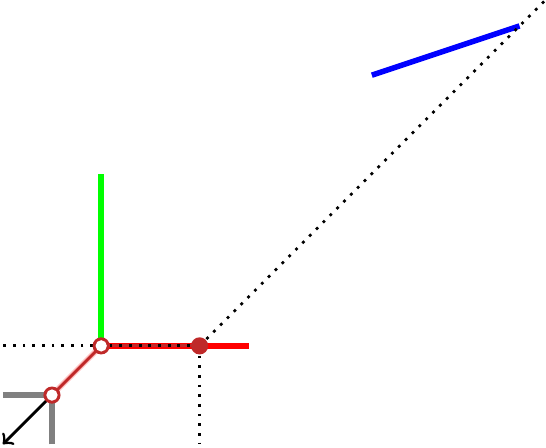}
		\caption*{(I)}
	\end{subfigure}
	\begin{subfigure}[b]{.15\textwidth}
		\centering\includegraphics[width=1\textwidth]{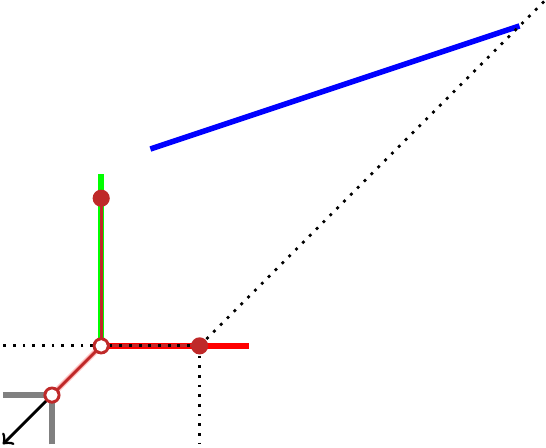}
		\caption*{(N)}
	\end{subfigure}
	\begin{subfigure}[b]{.15\textwidth}
		\centering\includegraphics[width=1\textwidth]{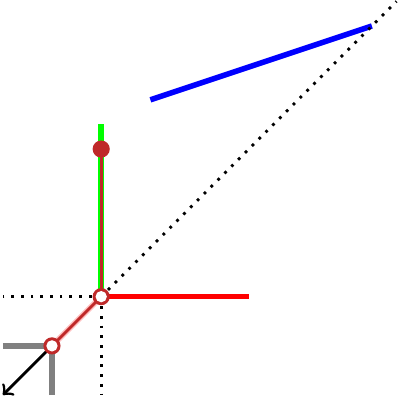}
		\caption*{(I)$_{(x\,y)}$}
	\end{subfigure}
	\begin{subfigure}[b]{.14\textwidth}
		\centering\includegraphics[width=1\textwidth]{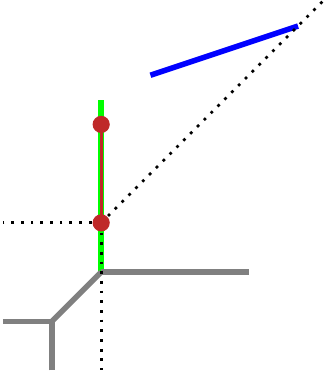}
		\caption*{(G)$_{(x\,y)}$}
	\end{subfigure}
	\caption{Example of deformation class  (G I N)$+(x\,y)$ with one tangency given by $\overline{p_{11}p_{04}}$.}\label{fig:subdivGINIG2}
\end{figure}

\newpage
\subsection{Deformation class (G K U T T')}
As for deformation class (G I N)$+(x\,y)$ there are two different cases of how the second tangency arises.
Let $\Gamma$ be a smooth tropical quartic curve  with a bitangent class $B$ 
with dual bitangent motif in identity position contained in the subcomplex  in Figure \ref{fig:subdivGKUT1}. 
This figure depicts the case where one tangency is given by the blue edge $\overline{p_{11}p_{22}}$.
\begin{figure}[h]
	\centering
	\begin{subfigure}[b]{.14\textwidth}
		\centering\includegraphics[width=0.9\textwidth]{figures/defclass/_G_-_K_-_U_-_T_-_T__1.pdf}
		\caption*{}
	\end{subfigure}
	\begin{subfigure}[b]{.18\textwidth}
		\centering\includegraphics[width=1\textwidth]{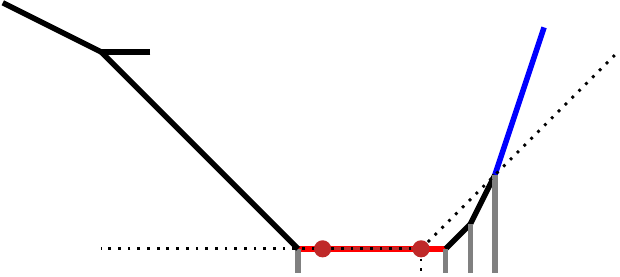}
		\caption*{(G)}
	\end{subfigure}
	\begin{subfigure}[b]{.17\textwidth}
		\centering\includegraphics[width=1\textwidth]{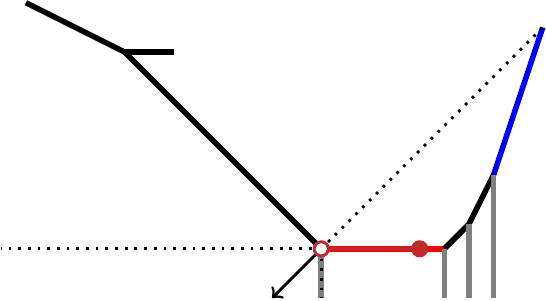}
		\caption*{(K)}
	\end{subfigure}
	\begin{subfigure}[b]{.16\textwidth}
		\centering\includegraphics[width=1\textwidth]{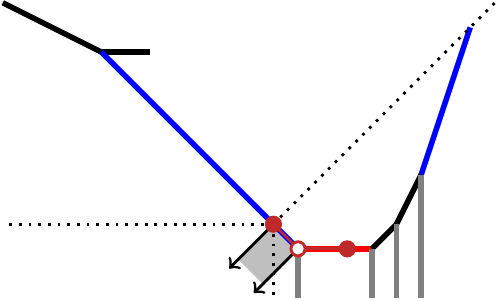}
		\caption*{(U)}
	\end{subfigure}
	\begin{subfigure}[b]{.15\textwidth}
		\centering\includegraphics[width=1\textwidth]{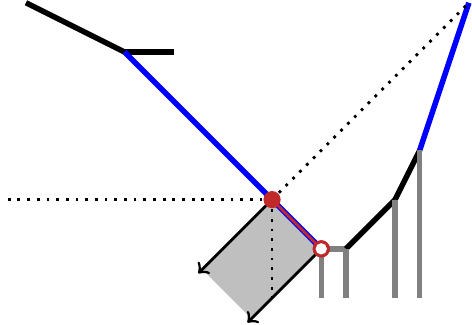}
		\caption*{(T')}
	\end{subfigure}
	\begin{subfigure}[b]{.14\textwidth}
		\centering\includegraphics[width=1\textwidth]{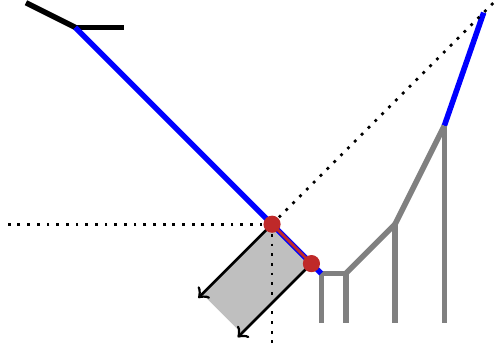}
		\caption*{(T)}
	\end{subfigure}
	\caption{Example of deformation class  (G K U T T') with one tangency given by $\overline{p_{11}p_{22}}$.}\label{fig:subdivGKUT1}
\end{figure}

\noindent The situation is similar if we choose the other blue edge $\overline{p_{11}p_{40}}$, as illustrated in Figure~\ref{fig:subdivGKUT2}.
\begin{figure}[h]
	\centering
	\begin{subfigure}[b]{.15\textwidth}
		\centering\includegraphics[width=0.9\textwidth]{figures/defclass/_G_-_K_-_U_-_T_-_T__2.pdf}
		\caption*{}
	\end{subfigure}
	\begin{subfigure}[b]{.18\textwidth}
		\centering\includegraphics[width=1\textwidth]{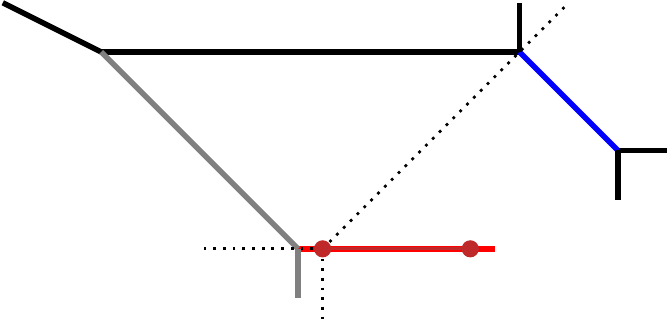}
		\caption*{(G)}
	\end{subfigure}
	\begin{subfigure}[b]{.17\textwidth}
		\centering\includegraphics[width=1\textwidth]{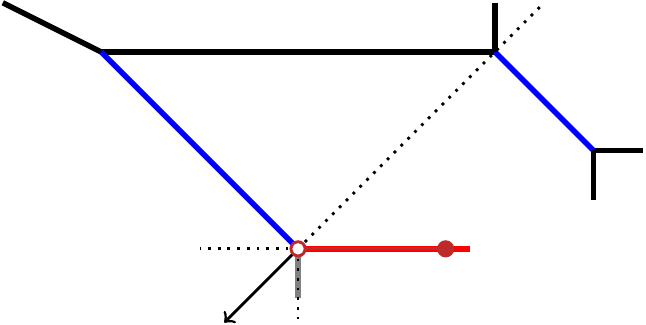}
		\caption*{(K)}
	\end{subfigure}
	\begin{subfigure}[b]{.16\textwidth}
		\centering\includegraphics[width=1\textwidth]{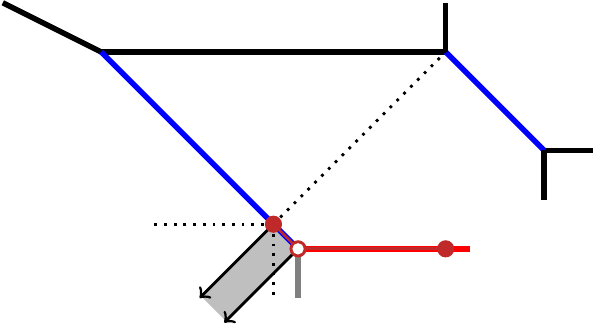}
		\caption*{(U)}
	\end{subfigure}
	\begin{subfigure}[b]{.15\textwidth}
		\centering\includegraphics[width=1\textwidth]{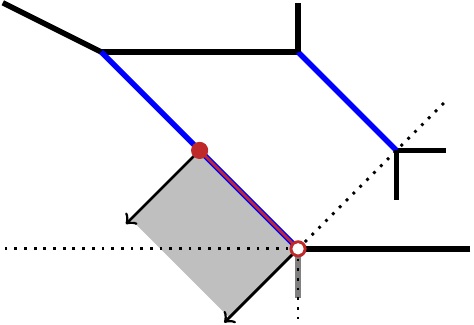}
		\caption*{(T')}
	\end{subfigure}
	\begin{subfigure}[b]{.14\textwidth}
		\centering\includegraphics[width=1\textwidth]{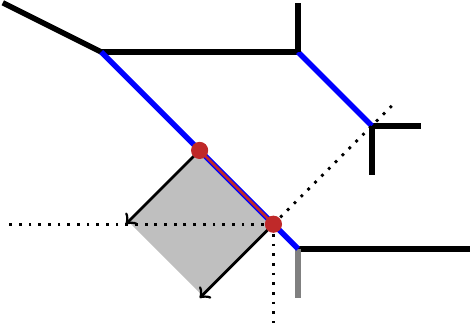}
		\caption*{(T)}
	\end{subfigure}
	\caption{Example of deformation class  (G K U T T') with one tangency given by $\overline{p_{11}p_{40}}$.}\label{fig:subdivGKUT2}
\end{figure}

\subsection{Deformation class (W X Y EE GG)}
Let $\Gamma$ be a smooth tropical quartic curve  with a bitangent class $B$ 
with dual bitangent motif in identity position contained in the subcomplex  in Figure \ref{fig:subdivWpart}. 
\begin{figure}[h]
	\centering
	\begin{subfigure}[b]{.14\textwidth}
		\centering\includegraphics[width=0.9\textwidth]{figures/defclass/_W_-_X_-_Y_-_EE_-_GG_.pdf}
		\caption*{}
	\end{subfigure}
	\begin{subfigure}[b]{.16\textwidth}
		\centering\includegraphics[width=1\textwidth]{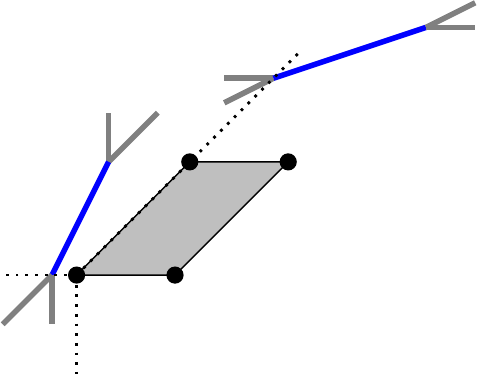}
		\caption*{(W)}
	\end{subfigure}
	\begin{subfigure}[b]{.16\textwidth}
		\centering\includegraphics[width=1\textwidth]{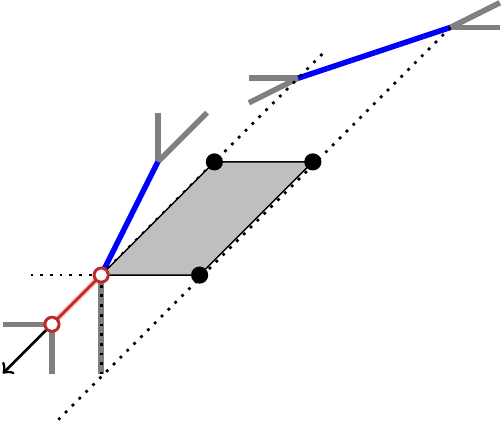}
		\caption*{(X)$_{(x\,z)}$}
	\end{subfigure}
	\begin{subfigure}[b]{.16\textwidth}
		\centering\includegraphics[width=1\textwidth]{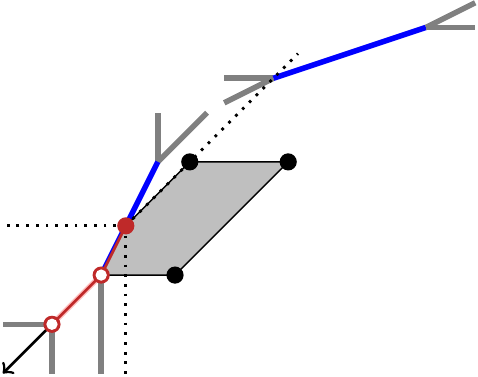}
		\caption*{(Y)$_{(x\,z)}$}
	\end{subfigure}
	\begin{subfigure}[b]{.16\textwidth}
		\centering\includegraphics[width=1\textwidth]{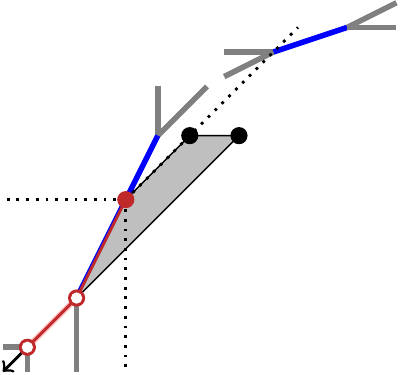}
		\caption*{(EE)}
	\end{subfigure}
	\begin{subfigure}[b]{.16\textwidth}
		\centering\includegraphics[width=1\textwidth]{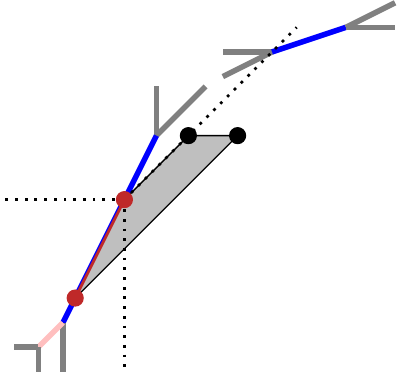}
		\caption*{(GG)}
	\end{subfigure}
	\caption{Example of deformation class  (W X Y EE GG).}\label{fig:subdivWpart}
\end{figure}

\subsection{Deformation class (W...HH)$+(x\,z)$}
Let $\Gamma$ be a smooth tropical quartic curve  with a bitangent class $B$ 
with dual bitangent motif in identity position contained in the subcomplex in Figure~\ref{fig:subdivWall}. 

\begin{figure}[h]
	\centering
\begin{subfigure}[b]{.2\textwidth}
	\centering\includegraphics[width=0.7\textwidth]{figures/defclass/_W_-to-_HH_.pdf}	
	\caption*{}
\end{subfigure}
\begin{subfigure}[b]{.24\textwidth}
	\centering\includegraphics[width=0.9\textwidth]{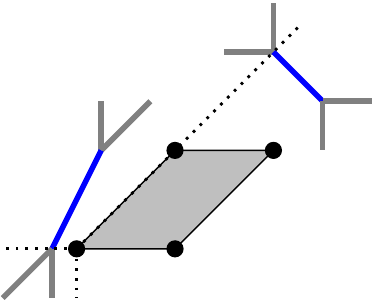}
	\caption*{(W)}
\end{subfigure}
\begin{subfigure}[b]{.24\textwidth}
	\centering\includegraphics[width=0.9\textwidth]{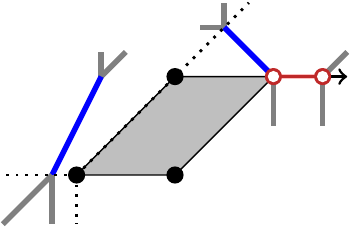}
	\caption*{(X)}
\end{subfigure}
\begin{subfigure}[b]{.24\textwidth}
	\centering\includegraphics[width=0.9\textwidth]{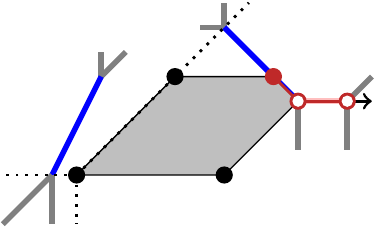}
	\caption*{(Y)}
\end{subfigure}
\begin{subfigure}[b]{.24\textwidth}
	\centering\includegraphics[width=0.9\textwidth]{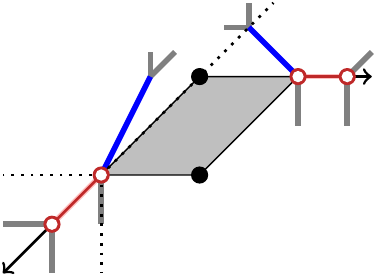}
	\caption*{(Z)}
\end{subfigure}
\begin{subfigure}[b]{.24\textwidth}
	\centering\includegraphics[width=0.9\textwidth]{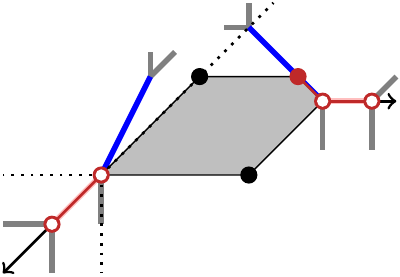}
	\caption*{(AA)}
\end{subfigure}
\begin{subfigure}[b]{.24\textwidth}
	\centering\includegraphics[width=0.9\textwidth]{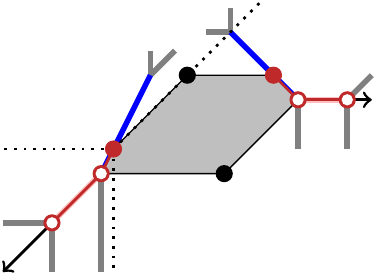}
	\caption*{(BB)}
\end{subfigure}
\begin{subfigure}[b]{.24\textwidth}
	\centering\includegraphics[width=0.75\textwidth]{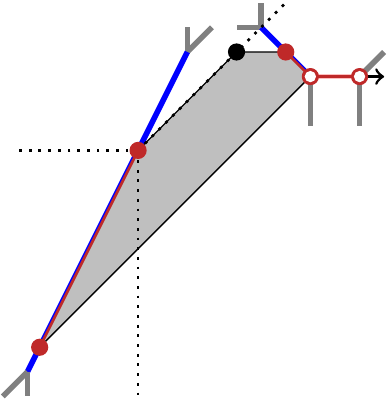}
	\caption*{(CC)}
\end{subfigure}
\begin{subfigure}[b]{.24\textwidth}
	\centering\includegraphics[width=0.75\textwidth]{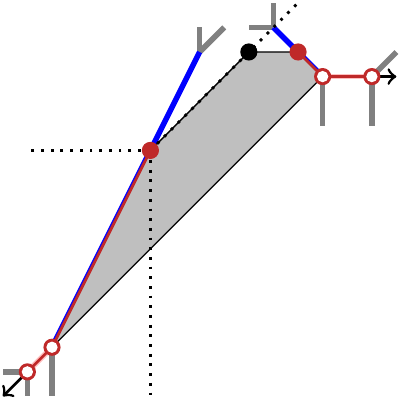}
	\caption*{(DD)}
\end{subfigure}
\begin{subfigure}[b]{.24\textwidth}
	\centering\includegraphics[width=0.75\textwidth]{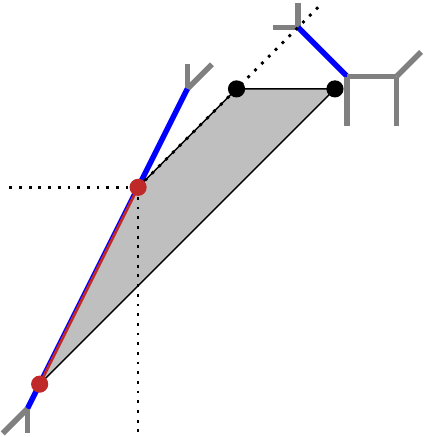}
	\caption*{(EE)}
\end{subfigure}
\begin{subfigure}[b]{.24\textwidth}
	\centering\includegraphics[width=0.75\textwidth]{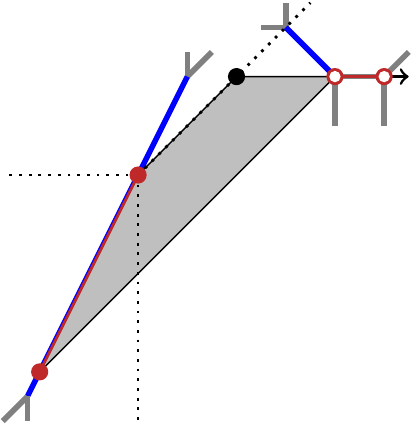}
	\caption*{(FF)}
\end{subfigure}
\begin{subfigure}[b]{.24\textwidth}
	\centering\includegraphics[width=0.75\textwidth]{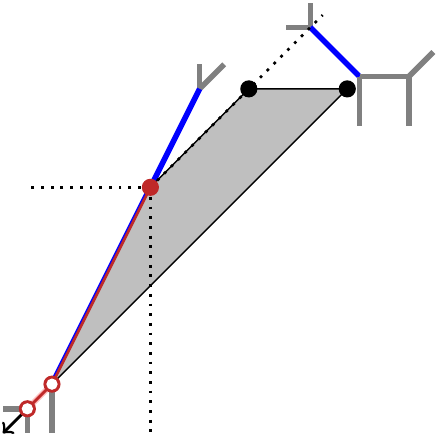}
	\caption*{(GG)}
\end{subfigure}
\begin{subfigure}[b]{.24\textwidth}
	\centering\includegraphics[width=0.75\textwidth]{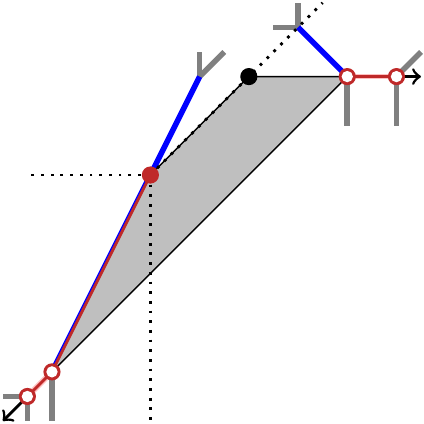}
	\caption*{(HH)}
\end{subfigure}
\begin{subfigure}[b]{.24\textwidth}
	\centering\includegraphics[width=0.9\textwidth]{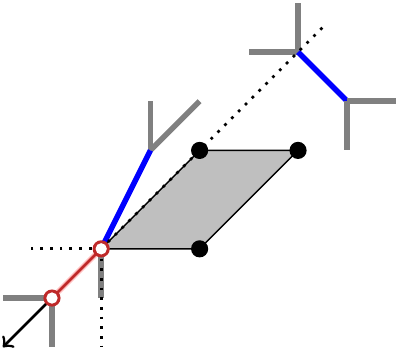}	
	\caption*{(X)$_{(x\,z)}$}
\end{subfigure}
\begin{subfigure}[b]{.24\textwidth}
	\centering\includegraphics[width=0.9\textwidth]{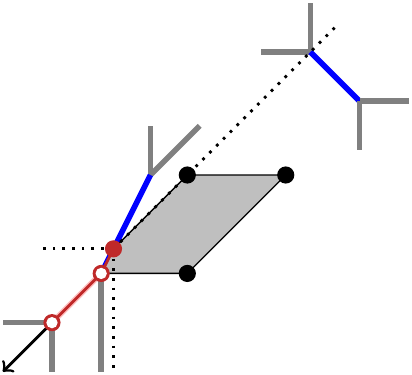}
	\caption*{(Y)$_{(x\,z)}$}
\end{subfigure}
\begin{subfigure}[b]{.24\textwidth}
	\centering\includegraphics[width=0.9\textwidth]{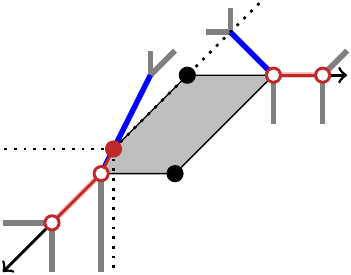}
	\caption*{(AA)$_{(x\,z)}$}
\end{subfigure}
\begin{subfigure}[b]{.24\textwidth}
	\centering\includegraphics[width=0.75\textwidth]{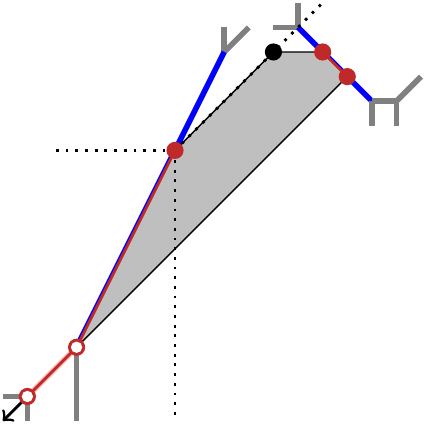}
	\caption*{(CC)$_{(x\,z)}$}
\end{subfigure}
\begin{subfigure}[b]{.24\textwidth}
	\centering\includegraphics[width=0.9\textwidth]{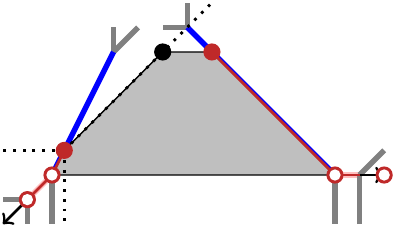}
	\caption*{(DD)$_{(x\,z)}$}
\end{subfigure}
\begin{subfigure}[b]{.24\textwidth}
	\centering\includegraphics[width=0.9\textwidth]{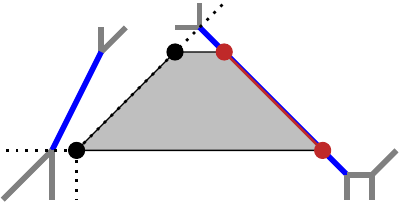}
	\caption*{(EE)$_{(x\,z)}$}
\end{subfigure}
\begin{subfigure}[b]{.24\textwidth}
	\centering\includegraphics[width=0.9\textwidth]{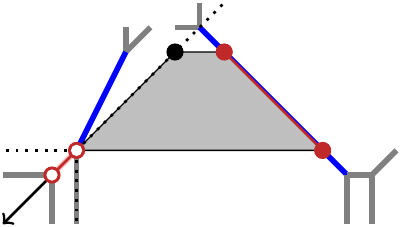}
	\caption*{(FF)$_{(x\,z)}$}
\end{subfigure}
\begin{subfigure}[b]{.24\textwidth}
	\centering\includegraphics[width=0.9\textwidth]{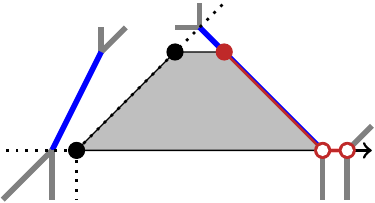}
	\caption*{(GG)$_{(x\,z)}$}
\end{subfigure}
\begin{subfigure}[b]{.24\textwidth}
	\centering\includegraphics[width=0.9\textwidth]{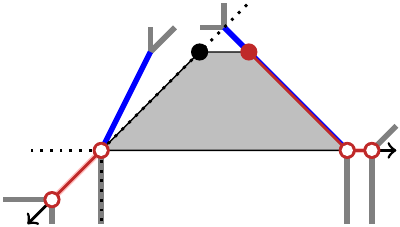}
	\caption*{(HH)$_{(x\,z)}$}
\end{subfigure}
	\caption{Example of deformation class  (W X Y Z AA ... HH).}\label{fig:subdivWall}
\end{figure}
\clearpage 
\bibliography{bibliography}
\end{document}